\documentclass{amsart}
\usepackage{amssymb}
\usepackage{amscd}
\usepackage{amsfonts}
\usepackage{latexsym}
\usepackage[all]{xy}
\usepackage{verbatim}

\newtheorem{theorem}{Theorem}[section]
\newtheorem{lemma}[theorem]{Lemma}
\newtheorem{proposition}[theorem]{Proposition}
\newtheorem{corollary}[theorem]{Corollary} 
\theoremstyle{definition}  
\newtheorem{definition}[theorem]{Definition}
\newtheorem{example}[theorem]{Example}

\newtheorem{remark}[theorem]{Remark}

\newcommand{\id}{\text{id}}

\newcommand{\Fun}{\text{Fun}}
\newcommand{\FPdim}{\text{FPdim}} 
\newcommand{\End}{\text{End}} 
\renewcommand{\Vec}{\text{Vec}}

\newcommand{\Hom}{\text{Hom}} 
\newcommand{\uHom}{\underline{\text{Hom}}}

\newcommand{\Rep}{\text{Rep}}

\newcommand{\B}{\mathcal{B}}
\newcommand{\C}{\mathcal{C}}
\newcommand{\D}{\mathcal{D}}
\newcommand{\E}{\mathcal{E}}

\newcommand{\Z}{\mathcal{Z}}

\newcommand{\M}{\mathcal{M}}

\newcommand{\N}{\mathcal{N}}

\newcommand{\be}{\mathbf{1}}

\renewcommand{\be}{\mathbf{1}}

\newcommand{\mB}{{\mathcal B}}
\newcommand{\mZ}{{\mathcal Z}}

\newcommand{\CC}{{\mathbb{C}}}

\newcommand{\bt}{\boxtimes}
\newcommand{\ot}{\otimes}

\begin{document}

\title{Weakly group-theoretical and solvable fusion categories}

\author{Pavel Etingof}
\address{P.E.: Department of Mathematics, Massachusetts Institute of Technology,
Cambridge, MA 02139, USA}
\email{etingof@math.mit.edu}

\author{Dmitri Nikshych}
\address{D.N.: Department of Mathematics and Statistics,
University of New Hampshire,  Durham, NH 03824, USA}
\email{nikshych@math.unh.edu}

\author{Victor Ostrik}
\address{V.O.: Department of Mathematics,
University of Oregon, Eugene, OR 97403, USA}
\email{vostrik@math.uoregon.edu}

\date{\today}
\maketitle  


\dedicatory{\bf To Izrail Moiseevich Gelfand on his 95th birthday
with admiration}

\section{Introduction and main results}

The goal of this paper is to introduce and study two classes of fusion
categories (over $\CC$): weakly
group-theoretical categories and solvable categories. 

Namely, recall (\cite{GNk}) that a fusion category $\C$ is said to
be {\em nilpotent} if there is a sequence of fusion categories 
$\C_0=\Vec,\, \C_1,\dots, \C_n = \C$
and a sequence $G_1, \dots G_n$ of finite groups 
such that $\C_{i}$ is obtained from $\C_{i-1}$ by a 
$G_i$-extension (i.e., $\C_i$ is faithfully graded by $G_i$ with
trivial component $\C_{i-1}$). Let us say that $\C$ is {\em cyclically
nilpotent} if the groups $G_i$ can be chosen to be cyclic (or,
equivalently, cyclic of
prime order). 

\begin{definition}\label{wgt}
A fusion category $\C$ is {\em weakly group-theoretical} if 
it is Morita equivalent to a nilpotent fusion category.
\footnote{It will be shown in a subsequent paper
that a weakly group-theoretical
category is in fact Morita equivalent to a nilpotent category of
nilpotency class $n=2$, i.e., to a group extension of a pointed
category. This should allow one to describe weakly group-theoretical
categories fairly explicitly in group-theoretical terms.}
\end{definition}

Here the notion of  {\em Morita equivalence} means the same as {\em weak monoidal Morita 
equivalence} introduced by M.~M\"uger in \cite{M2}. This is a categorical analogue of the familiar 
notion of Morita equivalence for rings.  

\begin{definition}\label{solv}
A fusion category 
$\C$ is {\em solvable}\footnote{This definition is motivated by the fact    
the category $\Rep(G)$ of representations of a finite group $G$ 
is solvable if and only if $G$ is a solvable group.} 
if any 
of the following two equivalent
conditions is
satisfied\footnote{The equivalence of these two conditions 
is proved in Proposition \ref{basicp2}.}: 
\begin{enumerate}
\item[(i)] $\C$ is Morita equivalent to a cyclically nilpotent fusion
category;
\item[(ii)] there is a sequence of fusion categories 
$\C_0=\Vec,\, \C_1,\dots, \C_n = \C$
and a sequence $G_1, \dots G_n$ of cyclic groups of prime order 
such that $\C_{i}$ is obtained from $\C_{i-1}$ either by a 
$G_i$-equivariantization or as a $G_i$-extension. 
\end{enumerate}
\end{definition}

Thus, the class of weakly group-theoretical categories contains 
the classes of solvable and group-theoretical categories
(i.e. those Morita equivalent to pointed categories, see \cite{ENO}
and Section~\ref{Sect 2.5} below). 
In fact, it contains all fusion categories we know which
are weakly integral, i.e., have integer Frobenius-Perron dimension.

Our first main result is the following characterization of fusion categories
Morita equivalent to group extensions of a given fusion category. 

\begin{theorem}
\label{Th0}
\label{Mor eq graded}
\label{Morita eq graded}
Let $\D$ be a fusion category and let $G$ be a finite group.
A fusion category $\C$ is Morita equivalent to a $G$-extension of $\D$ if and only if 
its Drinfeld center 
$\Z(\C)$ contains a Tannakian subcategory $\E=\Rep(G)$ such that
the de-equivariantization of $\E'$ by $\E$ is equivalent to $\Z(\D)$ as a braided tensor category.
\end{theorem}

The precise definition of de-equivariantization can be found in Section~\ref{Sect 2.6} below and 
in \cite[Section 4]{DGNO2}. In the context of modular categories it is the {\em modularization}
construction introduced by A.~Brugui\`{e}res \cite{B} and M.~M\"uger \cite{M3}.

In the special case when $G$ is the trivial group, Theorem~\ref{Th0}  simply says that two fusion 
categories are Morita equivalent if and only if their  
centers are equivalent as braided categories (Theorem~\ref{Morita eq iff}). 
This important result, which answers a question of V.~Drinfeld, has been 
announced by A.~Kitaev and M.~M\"uger,  and is used in the proof of the more 
general Theorem~\ref{Th0}. Since, as far as we know, a proof of this result is 
unavailable in the literature, we give such a proof in the beginning of 
Section~\ref{Sect 3}.  That Morita equivalent fusion categories
have braided equivalent centers was shown by M\"uger in \cite{M2}; 
we prove the opposite implication.

Our second main result may be viewed as a strong form of 
Kaplansky's 6-th conjecture (stating that 
the dimension of an irreducible representation of a semisimple
Hopf algebra divides the dimension of the Hopf algebra) for
weakly group-theoretical fusion categories. 
To state it, we need the following definition. 

\begin{definition}
\label{strong Frobenius}
We will say that a fusion category $\C$ has the {\em strong
Frobenius} property if for every indecomposable $\C$-module 
category $\M$ and any simple object $X$ in $\M$
the number  $\tfrac{\FPdim(\C)}{\FPdim(X)}$ is an algebraic integer,
where the Frobenius-Perron dimensions in $\M$
are normalized in such a way that $\FPdim(\M) =\FPdim(\C)$.
\end{definition}

Obviously, the strong Frobenius property of a fusion category
implies the usual Frobenius property, i.e. that the
Frobenius-Perron dimension of any simple object divides the
Frobenius-Perron dimension of the category (indeed, it suffices
to take $\M=\C$). 

\begin{theorem}\label{Th1}
Any weakly group-theoretical fusion category has the strong Frobenius
property. 
\end{theorem}

Finally, our third main result is an analog of Burnside's theorem for
fusion categories. 

\begin{theorem} \label{Th2} 
Any fusion category of Frobenius-Perron dimension $p^rq^s$, where
$p$ and $q$ are primes, and $r,s$ are nonnegative integers, is
solvable. 
\end{theorem}

Theorems \ref{Th1} and \ref{Th2} and the intermediate results
used in their proofs provide powerful methods for
studying weakly group-theoretical and solvable
fusion categories, in particular those of dimension $p^rq^s$, and
of (quasi)Hopf algebras associated to them. As an illustration,
we show that any non-pointed simple weakly group-theoretical fusion category 
is equivalent to $\Rep(G)$ for a finite non-abelian simple group $G$, 
and that any fusion category of
dimension 60 is weakly group-theoretical (in particular, if 
it is simple, it is equivalent to $\Rep(A_5)$). 
We also show that any fusion category of dimension $pqr$ (where $p,q,r$
are distinct primes) and any semisimple Hopf algebra 
of dimension $pqr$ or $pq^2$ is group-theoretical, and classify
such Hopf algebras. However, most of such applications
will be discussed in future publications. In particular,
the classification of fusion categories of dimension $pq^2$, 
where $q$ and $p$ are primes, is
given in the paper \cite{JL} 
(as pointed out in \cite{ENO}, not all such categories are
group-theoretical, already for $p=2$).  

The structure of the paper is as follows. In Section~\ref{Sect 2}, 
we discuss preliminaries. In Section~\ref{Sect 3} we 
prove Theorem~\ref{Th0}. 
In Section~\ref{Sect 4}, we state and prove the basic 
properties of weakly group-theoretical and solvable fusion categories. 
In Section~\ref{Sect 5},  we describe module categories over
equivariantizations; this description 
needed for the proof of Theorem \ref{Th1}. 
In Section~\ref{Sect 6}, we prove Theorem \ref{Th1}. 
In Section~\ref{Sect 7}, we prove some important properties of non-degenerate and
slightly degenerate categories containing a simple object of prime
power dimension, which are needed for the proof of Theorem
\ref{Th2}. In particular, Corollary~\ref{primpow} is similar to the classical
Burnside Lemma in the theory of group representations.
In Section~\ref{Sect 8}, we prove Theorem \ref{Th2}. 
In Section~\ref{Sect 9}, we discuss applications of our results to 
concrete problems in the theory of fusion categories 
and Hopf algebras. In Section~\ref{Sect 10} we briefly discuss the relation 
between our results and classification
results on semisimple Hopf algebras and fusion categories 
available in the literature.
Finally, in Section~\ref{Sect 11} we formulate some open questions. 

{\bf Acknowledgments.} 
We are deeply grateful to V.~Drinfeld for many inspiring
conversations. Without his influence,  
this paper would not have been written. 
We also thank S.~Gelaki and D.~Naidu for useful discussions. 
The work of P.E.\ was  partially supported
by the NSF grant DMS-0504847.  The work of D.N.\ was partially supported 
by the NSA grant H98230-07-1-0081 and the NSF grant DMS-0800545.
The work of V.O.\ was  partially supported
by the NSF grant DMS-0602263.

\section{Preliminaries} 
\label{Sect 2}
In this paper, we will freely use the basic theory of fusion categories, module categories over them, 
Frobenius-Perron dimensions, and modular categories. For basics
on these topics, we refer the reader to  
\cite{BK,O1,ENO, DGNO2}.\footnote{All fusion categories and (quasi)Hopf
algebras in this paper
will be over $\Bbb C$ (which can be replaced by any algebraically
closed field of characteristic zero). 
All module categories will be left module categories. For a fusion category $\C$ we use 
notation $\Z(\C)$ for its Drinfeld center 
(see e.g. \cite[\S 2.9]{DGNO2}).}
However, for reader's convenience, we recall
some of the most important definitions and facts 
that are used below.

\subsection{Graded categories} (\cite{ENO,GNk})

Let $\C$ be a fusion category and let $G$ be a finite group.
We say that $\C$ is {\em graded} by $G$ if $\C=\bigoplus_{g\in G}\C_g$,
and for any $g,h\in G$, one has 
$\otimes: \C_g\times \C_h\to \C_{gh}$, $*:\C_g\to
\C_{g^{-1}}$. The fusion category $\C_e$ corresponding to the neutral element
$e\in G$ is called {\em the trivial component} of the $G$-graded
category $\C$. A grading is {\em faithful} if $\C_g\ne 0$ for all $g\in
G$. If $\C$ is faithfully graded by $G$, one says that $\C$ is a $G$-{\em
extension} of $\C_e$. The {\em adjoint category} $\C_{\rm ad}$ is the smallest
fusion subcategory of $\C$ containing all objects $X\otimes X^*$, where
$X\in \C$ is simple. 

There exists a unique faithful grading of $\C$ for which $\C_e=\C_{\rm ad}$. 
It is called {\em the universal grading} of $\C$. The
corresponding group is called {\em the universal grading group}
of $\C$, and denoted by $U_\C$. All faithful gradings of $\C$ are
induced by the universal grading, in the sense that for any
faithful grading  $U_\C$ canonically projects onto the grading group $G$, and
$\C_e$ contains $\C_{\rm ad}$. 

A fusion category $\C$ is said to be {\em nilpotent} if 
it can be reduced to the category of vector spaces by iterating
the operation of taking the adjoint category. This is equivalent
to the condition that $\C$ can be
included into a chain ${\rm Vec}=\C_0\subset
\C_1\subset...\subset \C_n=\C$, where each $\C_i$ is faithfully graded by
a finite group $G_i$, and has trivial component $\C_{i-1}$. 

The simplest example of a nilpotent category is a {\em pointed}
category, i.e. a fusion category where all simple objects are
invertible. Such a category is the category of vector
spaces graded by some finite group $G$ 
with associativity defined by a cohomology class
$\omega\in H^3(G,\Bbb C^*)$, denoted by ${\rm
Vec}_{G,\omega}$. If $\omega=1$, we denote this category by ${\rm
Vec}_G$. 

\subsection{Frobenius-Perron dimensions in a module category}   

Let $\C$ be a fusion category, and $\M$ an indecomposable module
category over $\C$. Let $M_i, i\in I,$ be the simple objects of $\M$. 
Then it follows from the Frobenius-Perron
theorem that there exists a unique, up to a common factor, 
collection of positive numbers $d_i, i\in I$, such that 
whenever $X\in \C$, and $X\otimes M_i=\oplus_{j\in I} N_{ij}(X)M_j$, 
one has ${\rm FPdim}(X)d_i=\sum_{j\in I} N_{ij}(X)d_j$. 
We will normalize $d_i$ in such a way that $\sum_{i\in
I}d_i^2={\rm FPdim}(\C)$. The numbers $d_i$ normalized in such a
way are called {\rm the Frobenius-Perron
dimensions} of $M_i$. By additivity, this defines the
Frobenius-Perron dimension of any object of~$\M$. 

\subsection{Weakly integral and integral categories} 

A fusion category is said to be {\em weakly integral}, 
if its Frobenius-Perron dimension is an integer. Recall
\cite[Proposition 8.27]{ENO} that in such a category, the Frobenius-Perron 
dimension of any simple object is the square root of an
integer. 

A fusion category is called {\em integral} if the Frobenius-Perron dimension 
of every (simple) object is an integer. A weakly  integral fusion category is
automatically pseudounitary and has a canonical spherical structure
with respect to which categorical dimensions coincide with the Frobenius-Perron 
dimensions \cite[Propositions 8.23, 8.24]{ENO}.


\subsection{Tannakian categories}

Recall (\cite{D1}) that a symmetric fusion category $\C$ is {\em Tannakian} 
if it is equivalent to the representation category of a finite
group as a symmetric fusion category. More generally, let us say
that $\C$ is {\em super-Tannakian} if there exists a finite group $G$ and a
central element $u\in G$ of order 2, such that $\C$, as a
symmetric category, is equivalent to the category of
representations of $G$ on super vector spaces, 
on which $u$ acts by the parity operator. 

\begin{theorem}(Deligne's 
theorem, \cite{D2})\label{DeT} Any symmetric fusion category is
super-Tannakian. 
\end{theorem}

In particular, if $\C$ has Frobenius-Perron
dimension bigger than $2$, then it contains a nontrivial Tannakian 
subcategory (the category of representations of $G/(u)$). 

\subsection{Morita equivalence} (\cite{M2}; see also \cite{ENO, O1})
\label{Sect 2.5}

One says that two fusion categories $\C$ and $\D$ are 
{\em Morita equivalent} if $\D\cong (\C_\M^*)^{op}$ for some 
indecomposable $\C$-module category $\M$ (the category of 
$\C$-module endofunctors of $\M$, with opposite composition).
Equivalently, there is an algebra $A$ in $\C$ such that $\D$ is equivalent
to the category of $A$-bimodules in $\C$.

The above is an equivalence relation on fusion categories of a given
Frobenius-Perron dimension. A fusion category is said to be 
{\it group-theoretical} if it is Morita equivalent to a pointed
category. A (quasi)Hopf algebra is group-theoretical if
its representation category is group-theoretical. 

\subsection{Equivariantization and de-equivariantization (\cite{B, M3} and \cite[Section 4]{DGNO2})}
\label{Sect 2.6}

Let $\C$ be a fusion category with an action of a finite 
group $G$. In this case one can define the fusion category $\C^G$ of 
$G$-equivariant objects in $\C$. 
Objects of this category are objects $X$ of $\C$ equipped with 
an isomorphism $u_g: g(X)\to X$ for all $g\in G$, such that
$$
u_{gh}\circ \gamma_{g,h}=u_g\circ g(u_h),
$$
where $\gamma_{g,h}: g(h(X))\to gh(X)$ is the natural isomorphism
associated to the action. Morphisms and tensor
product of equivariant objects 
are defined in an obvious way. This category is called
the {\em $G$-equivariantization} of $\C$. One has 
$\FPdim(\C^G)=|G|\FPdim(\C)$. 

For example, $\Vec^G=\Rep(G)$ (for the trivial action of $G$
on $\Vec$). A more interesting example is the following. 
Let $K$ be a normal subgroup of $G$. Then we have a natural action of $G/K$ on
$\Rep(K)$, and $\Rep(K)^{G/K}=\Rep(G)$.

There is a procedure opposite to equivariantization, called the 
{\em de-equivariantiza\-tion}.
In the context of modular categories it is the {\em modularization}
construction introduced by A.~Brugui\`{e}res and M.~M\"uger, see Remark~\ref{modularization}
below.  It is also closely related to the dynamical extensions of monoidal
categories of J.~Donin and A.~Mudrov \cite{DM}.

Namely, let $\C$ be a fusion category  and let
$\E= \Rep(G) \subset \Z(\C)$  be a Tannakian subcategory which embeds into $\C$
via the forgetful functor $\Z(\C)\to \C$. 
Let $A$ be the algebra in $\Z(\C)$ corresponding to the algebra
$\mbox{Fun}(G)$ of functions on $G$ under the above embedding.
It is a commutative algebra in $\Z(\C)$ and so
the category $\C_G$ of left $A$-modules in $\C$ is a
fusion category, called {\em de-equivariantization} of $\C$ by $\E$. 
The  free module functor $\C \to \C_G: X \mapsto A\ot X$ is
a surjective tensor functor. One has $\FPdim(\C_G) = \FPdim(\C)/ |G|$.

The above constructions are canonically inverse to each other, i.e., there are
canonical equivalences $(\C_G)^G \cong \C$ and $(\C^G)_G \cong \C$. 
See  \cite[Proposition 4.19]{DGNO2}.

\subsection{The crossed product fusion category}\label{cro}

Let $\C$ be a fusion category, and $G$ a finite group acting on
$\C$. Then the {\em crossed product category} $\C \rtimes G$ is defined as
follows \cite{T}. 

For a pair of Abelian categories $\mathcal{A}_1$,  $\mathcal{A}_2$,
let $\mathcal{A}_1 \bt \mathcal{A}_2$ denote their Deligne's
tensor product \cite{D1}\footnote{For semisimple categories $\mathcal{A}_i$,
the Deligne tensor product 
$\mathcal{A}_1 \bt \mathcal{A}_2$ is just the category whose
simple objects are $X_1\otimes X_2$, where $X_i\in \mathcal{A}_i$ are
simple.}. We set $\C \rtimes G = \C \bt \Vec_G$ as an
abelian category, and define a tensor product by
\begin{equation}
(X \bt g) \ot (Y \bt h): = (X \ot g(Y)) \bt gh,\qquad X,Y \in
\C,\quad g,h\in G.
\end{equation}
Then $\M=\C$ is naturally a module category over $\C\rtimes G$
and we have the following result.
\begin{proposition} (see  \cite[Proposition 3.2]{Nk2})
\label{equiv}
$\C^G \cong (\C\rtimes G)_\M^{*op}$. 
\end{proposition}
In other words, the crossed product category $\C\rtimes G$ is
dual to the equivariantization $\C^G$ with respect to the $\C^G$-module
category $\C$. 

\subsection{The M\"uger centralizer} (\cite{M4})

Let $\C$ be a braided fusion category, and $\D\subset \C$ a full
subcategory. The {\em M\"uger centralizer} 
$\D'$ of $\D$ in $\C$ 
is the category of all objects $Y\subset \C$ such that for
any $X\subset \D$ the squared braiding on $X\otimes Y$ is
the identity. 
The {\em M\"uger center} of $\C$
is the M\"uger centralizer $\C'$ of the entire category $\C$. 
The category $\C$ is {\em non-degenerate} 
(in the sense of M\"uger) if $\C'=\Vec$. If $\C$ is a non-degenerate  
braided fusion category 
then one has $\FPdim(\D)\FPdim(\D')=\FPdim(\C)$
(see \cite[Theorem 3.2]{M4} and \cite[Theorem 3.14]{DGNO2}).
If $\C$ is non-degenerate then $\D\subset \C$ is called {\em Lagrangian} if
$\D=\D'$.          

\begin{remark} (see \cite{B, M3})
\label{modularization}
Let $\C$  be a braided fusion category and let $\E \subset \C'$ be a Tannakian
subcategory. Then the de-equivariantization 
of $\C$ by $\E$
is a braided fusion category. It is non-degenerate if and only if $\E=\C'$. 
\end{remark}
 
Note that if $\C$ is a weakly integral non-degenerate category, 
then by a result of \cite{ENO}, it is pseudounitary, which
implies that it is canonically a modular category. This fact will
be used throughout the paper. 

\subsection{M\"uger's theorem (\cite[Theorem 4.2]{M4},  \cite[Theorem 3.13]{DGNO2})}

\begin{theorem}\label{MuTh} 
Let $\C$ be a braided category, and $\D$ a non-degenerate
subcategory in $\C$. Then $\C$ is naturally equivalent, as a
braided category, to $\D\boxtimes \D'$. 
\end{theorem}

\subsection{Slightly degenerate categories}

\begin{definition} A braided fusion 
category $\C$ is called {\em slightly degenerate} if its M\"uger
center $\C'$ is equivalent, as a symmetric category, 
to the category SuperVec of super
vector spaces. 
\end{definition}

\begin{proposition}\label{almnon}
(i) {\em (cf. \cite[Lemma 5.4]{M1})} Let $\C$ be a braided fusion category such that
its M\"uger center $\C'$ contains ${\rm SuperVec}$
(for example, a slightly degenerate category).
Let $\chi$ be the invertible object generating 
${\rm SuperVec} \subseteq \C'$, 
and let $Y$ be any simple object of $\C$. Then 
$\chi\otimes Y$ is not isomorphic to $Y$. 

(ii) Let $\C$ be slightly degenerate and pointed. 
Then $\C={\rm SuperVec}\boxtimes \C_0$, where 
$\C_0$ is a non-degenerate pointed category.  
\end{proposition}

\begin{proof} (i) Assume the contrary, i.e., 
$\chi\otimes Y=Y$. Since $\chi$ centralizes $Y$, we have  
from this identity that the trace $T$ of the Drinfeld  
isomorphism $u: Y\to Y^{**}$ is equal to $-T$ 
(as $u|_\chi=-1$).  
This is a contradiction, as $T\ne 0$.

(ii) This statement is proved in \cite[Corollary A.19]{DGNO2}.
We provide a proof here for the reader's convenience. 
Our job is to show that $\chi\ne \xi^{\otimes 2}$ for any $\xi$
(this is the condition for the group of invertible objects of $\C$ to be the direct 
product of the $\Bbb Z/2\Bbb Z$ generated by $\chi$ with another subgroup). 
Assume the contrary, and let $Q$ be the quadratic form defining the braiding. 
Then we have $Q(\xi)^4=Q(\chi)=-1$, $Q(\chi\otimes \xi)=Q(\xi^{\otimes 3})=Q(\xi)^9=Q(\xi)$,
so the squared braiding of $\xi$ and $\chi$ is 
$$
\beta_{\chi,\xi}=Q(\chi\otimes \xi)/Q(\chi)Q(\xi)=-1, 
$$ 
which is a contradiction with the centrality of $\chi$. 
\end{proof}

Note that if $\C$ is a weakly integral 
braided category, then it is pseudounitary, and hence 
is canonically a ribbon category (\cite{ENO}).  By the {\em $S$-matrix}
of a ribbon category $\C$ we understand  a square matrix $S:=\{ s_{X,Y}\}$
whose columns and rows are labeled by simple objects of $\C$ and the entry
$s_{X,Y}$ is equal to the quantum trace of 
$c_{Y,X}c_{X,Y} : X\ot Y \to X\ot Y$,
where $c$ denotes the braiding of $\C$.

\begin{corollary}\label{ones} If $\C$ is a weakly integral slightly degenerate
braided category, then the $S$-matrix of $\C$ is  
$S=I\otimes S'$, where $S'$ is a non-degenerate matrix with
orthogonal rows and columns, and $I$ is the 2 by 2 matrix 
consisting of ones.
\end{corollary}

\begin{proof} This is an easy consequence of Proposition \ref{almnon}(i)
since the rows and columns of $S$ corresponding  to $\be$ and $\chi$ coincide.
\end{proof} 

\begin{corollary}\label{odd}
Let $\C$ be a slightly degenerate integral category of dimension $>2$.
Then $\C$ contains an odd-dimensional simple object 
outside of the M\"uger center $\C'$ of $\C$. 
\end{corollary}

\begin{proof}
Let $\chi$ be the invertible object generating $\C'$. 
Let $X$ be any simple object outside of $\C'$. By Proposition
\ref{almnon}(i), $\chi\otimes X\ne X$, which implies that 
$X\otimes X^*$ does not contain $\chi$. Thus, either $X$ itself 
is odd-dimensional, or $X\otimes X^*/\bold 1$ is odd-dimensional, 
and is a direct sum of simple objects not contained in $\C'$.
In this case one of the summands has to be odd-dimensional.
\end{proof}

\subsection{Interpretation of extensions and equivariantizations in terms of  
the  center (see \cite{B, M1, M3, DGNO2, GNN})}

\begin{proposition}
\label{eqext1}
Let $\C$ be a fusion category. 
\begin{enumerate}
\item[(i)]  If $\Z(\C)$ contains a  Tannakian subcategory $\Rep(G)$ which maps to $\Vec$
under the forgetful functor $\Z(\C)\to \C$ then $\C$ is a $G$-extension of some  
fusion category $\D$.
\item[(ii)]  Let $\C$ be a $G$-extension of $\D$. Then $\Z(\C)$ contains
a Tannakian subcategory $\E=\Rep(G)$ mapping to $\Vec$ in $\C$, 
such that the de-equivariantization 
of $\E'$ by $\E$ is equivalent to $\Z(\D)$ as a braided tensor category.
\end{enumerate}
\end{proposition}
\begin{proof}
(i)  Suppose there is a Tannakian subcategory $\E=\Rep(G) \subset \Z(\C)$
such that  the restriction of the forgetful functor $F: \Z(\C)\to \C$ maps
$\E$ to $\Vec$.
Then every simple object $X$ of $\C$ determines a tensor 
automorphism of $F|_\E$ as follows.
Given  an object $Y$ in $\E$, the permutation isomorphism 
$\eta_{X,Y}: X\otimes F(Y)\xrightarrow{\sim} F(Y)\otimes X$ defining the
central structure of $Y$ yields an automorphism 
$\eta_{X,Y}\circ \delta$ of $F(Y) \ot X$, where $\delta:
F(Y)\otimes X\to X\otimes F(Y)$ is the ``trivial'' isomorphism, coming
from the fact that $F(Y)\in \Vec$. Since $\End_\C(F(Y)\otimes X)=
\End_{\Bbb C}F(Y)$, we obtain a linear automorphism 
$i_X : F(Y)\to F(Y)$.  Clearly, $i_X$ gives rise
to a tensor automorphism of $F|_\E$. 
Since the group of tensor automorphisms of $F|_\E$ is 
isomorphic to $G$, we have a canonical assignment  $X \mapsto i_X \in G$.
It is multiplicative in $X$ (in the sense that $i_Z=i_Xi_Y$ for
any simple $Z\subset X\otimes Y$), and thus 
defines a grading $\C =\bigoplus_{g\in G}\, \C_g$. 

Now note that every simple object of the center $\Z(\C)$ of a graded
category $\C$ is either supported on its trivial component
or is disjoint from it.  By construction, $\E'$ coincides
with the category $\Z(\C)_e$ of objects of $\Z(\C)$ supported on $\C_e$
(indeed, $X$ is in $\E'$ iff $i_X$ is identity).
Therefore, $F$ restricts to a surjective functor  $\E'\to \C_e$.  
Using the identity in \cite[proof of Corollary 8.11]{ENO} we obtain
\[
\FPdim(\C_e) =\frac{\FPdim(\E')}{\FPdim(\C)} = \frac{\FPdim(\C)}{|G|},
\]
which means that the above grading of $\C$ is faithful.

(ii)  This statement is proved in \cite{GNN},
we include its proof for the reader's convenience.
Suppose $\C =\bigoplus_{g\in G}\, \C_g$ with $\C_e=\D$.
We construct a subcategory $\E \subset \Z(\C)$ as follows.
For any representation $\pi : G \to GL(V)$ of $G$ consider 
an object $Y_\pi$ in $\Z(\C)$
where $Y_\pi = V \ot \be$ as an object of $\C$
with the permutation isomorphism 
\begin{equation}
c_{Y_\pi, X}:= \pi(g) \ot \id_X:
 Y_\pi \ot X \cong X \ot Y_\pi,\qquad \mbox{ when }  
X\in \C_g,
\end{equation}
where we identified $Y_\pi \ot X$ and $X \ot Y_\pi$ with $V \ot X$.
Let $\E$ be the fusion subcategory of $\Z(\C)$ consisting of  objects $Y_\pi$, 
where $\pi$ runs through all finite-dimensional representations of $G$. 
Clearly, $\E$ is equivalent to  $\Rep(G)$ with its standard
braiding. 

By construction, the forgetful functor maps $\E$ to $\Vec$ and  $\E'$ 
consists of all objects in $\Z(\C)$ whose forgetful image is in $\C_e$.   
Consider the surjective braided functor $H: \E' \to \Z(\C_e)$ obtained by 
restricting the braiding of $X\in \E'$ from $\C$ to $\C_e$. One can check 
that $H$ can be factored through the de-equivariantization functor $\E' \to \E'_G$
(see \cite[Theorem 3.1]{B}). This yields a braided equivalence between $\E'_G$ 
and $\Z(\C_e)$, since the two categories have equal Frobenius-Perron dimension.
\end{proof}

The following Proposition \ref{eqext2} can be derived from \cite{B, M1, M2}.
Again, we include the proof for the reader's convenience.

\begin{proposition}
\label{eqext2}
\begin{enumerate}
\item[(i)]
If $\Z(\C)$ contains a  Tannakian subcategory $\Rep(G)$ which embeds  to $\C$
under the forgetful functor $\Z(\C)\to \C$ then $\C$ is a $G$-equivarianti\-zation of some  
fusion category $\D$.
\item[(ii)]
Let $\C$ be a $G$-equivariantization of $\D$. Then $\Z(\C)$ contains
a Tannakian subcategory $\E=\Rep(G)$  such that the de-equivariantization 
of $\E'$ by $\E$ is equivalent to $\Z(\D)$ as a braided tensor category.
\end{enumerate}
\end{proposition}
\begin{proof}
(i) Let $A=\mbox{Fun}(G)$ be the algebra of functions on $G$. It is a commutative
algebra in $\Z(\C)$. Therefore, the category $\D:=\C_G$ of $A$-modules in $\C$ is a 
fusion category. The action of $G$ on $A$ via right translations gives rise
to an action of $G$ on $\D$. It is straightforward to check 
that the corresponding equivariantization of $\D$ is equivalent to $\C$ 
(see Section~\ref{Sect 2.6} and \cite[Section 4.2]{DGNO2}).

(ii) By Proposition~\ref{equiv} 
$\D^G$ is Morita equivalent to a
$G$-graded fusion category $\D\rtimes G$ whose trivial component is $\D$. Hence
$\Z(\C)\cong \Z(\D \rtimes G)$ and
the required statement follows immediately from Proposition~\ref{eqext1}(ii).
\end{proof}

%
%
%
%

\subsection{The divisibility theorems} 

\begin{theorem}\label{divis} (\cite[Lemma 1.2]{EG1} and \cite[Propositions 8.23, 8.24, 8.27]{ENO})
\begin{enumerate}
\item[(i)] Let $\C$ be a weakly integral non-degenerate braided category. 
Then for any simple object $X\in \C$, the ratio
$\FPdim(\C)/\FPdim(X)^2$ is an integer. 
\item[(ii)] Let $\C$ be a weakly integral braided fusion category. 
Then for any simple object $X\in \C$, the ratio
$\FPdim(\C)/\FPdim(X)$ is the square root of an integer. 
\end{enumerate}
\end{theorem}

\begin{theorem}\label{divis1} (\cite[Corollary 8.11, Proposition 8.15]{ENO})
The Frobenius-Perron dimension of a full fusion subcategory or 
a component in a quotient category of a fusion category $\C$ divides 
the Frobenius-Perron dimension of $\C$ in the ring of algebraic integers. 
\end{theorem}

\subsection{Kac algebras (abelian extensions) \cite{N1}}

Let $G$ be a finite group, and $G=KL$ be an exact factorization
of $G$ into a product of two subgroups $K,L$ (exactness means
that $K\cap L=1$). Let $\omega\in Z^3(G,\Bbb C^*)$ be a
3-cocycle, and $\psi: C^2(K,\Bbb C^*),\phi\in C^2(L,\Bbb C^*)$ 
be 2-cochains such that $d\psi=\omega|_K,d\phi=\omega|_L$. 
Consider the fusion category $\C=\Vec_{G,\omega,K,\psi}$ of
$(K,\psi)$-biequivariant $(G,\omega)$-graded vector spaces
(see \cite{O1,O2}). 
This category has a module category $\M$ of $(G,\omega)$-graded
vector spaces which are equivariant under $(K,\psi)$ on the left,
and under $(L,\phi)$ on the right. This module category has only
one simple object, so it defines a fiber functor on $\C$, hence a
group-theoretical Hopf algebra $H=H(G,K,L,\omega,\psi,\phi)$ 
with $\Rep(H)=\C$. This Hopf algebra is an
{\it abelian extension} 
$$
\Bbb C\to {\rm Fun}(L)\to H\to \Bbb C[K]\to \Bbb C, 
$$
and is sometimes called a {\it Kac algebra}.

\section{Proof of Theorem~\ref{Th0}}
\label{Sect 3}

We start with the following characterization of Morita equivalence
of fusion categories in terms of their centers.

\begin{theorem}
\label{Morita eq iff}
Two fusion categories $\C$ and $\D$ are Morita equivalent if and only if
$\Z(\C)$ and $\Z(\D)$ are equivalent as braided tensor categories.
\end{theorem}
\begin{proof}
By the result of M\"uger \cite[Remark 3.18]{M2} Morita equivalent fusion categories
have braided equivalent centers. Thus we need to prove the opposite implication.

Given an algebra $A$ in a fusion category $\C$ let $A-\mbox{mod}_\C$ and $A-\mbox{bimod}_\C$ 
denote, respectively, the categories of right $A$-modules and $A$-bimodules in $\C$.
In the case when the category $\C$ is braided and the algebra $A$ is commutative,
the category $A-\mbox{mod}_\C$ has a natural structure of tensor category, see e.g. \cite{B}. Namely 
any $M\in A-\mbox{mod}_\C$ can be turned into $A-$bimodule using the morphism
$A\ot M\xrightarrow{c_{A,M}}M\ot A\to M$ and the tensor product on 
$A-\mbox{mod}_\C$ is defined to be tensor product $\ot_A$ over $A$, see e.g. \cite{O1}.

For a fusion category $\C$ let $F_\C: \Z(\C)\to \C$ and $I_\C: \C \to \Z(\C)$
denote the forgetful functor and its right adjoint.
The following Lemma is a special case of a more general result obtained in \cite{DMNO}.

\begin{lemma}  (i) The object $A=I_\C(\be)\in \Z(\C)$ has a natural
structure of commutative algebra; moreover for any $X\in \C$ the object $I_\C(X)$ has
a natural structure of right $A-$module.

(ii) The functor $I_\C$ induces an equivalence of tensor categories $\C \simeq A-\mbox{mod}_{\Z(\C)}$.
\end{lemma}

\begin{proof} Consider the category $\C$ as a module category over $\Z(\C)$ via the functor
$F_\C$. Then \cite[Lemma 3.38]{EO} says that $I_\C(X)=\uHom(\be,X)$ for any $X\in \C$. Thus
$A=I_\C(\be)=\uHom(\be, \be)$ has a natural structure of algebra in $\Z(\C)$; for any
$X\in \C$ the object $I_\C(X)=\uHom(\be,X)$ is naturally right $A-$module and the functor
$I_\C(?)=\uHom(\be, ?)$ induces an equivalence of categories $\C \simeq A-\mbox{mod}_{\Z(\C)}$,
see \cite[Theorem 3.17]{EO}. It remains to explain that $A$ is a commutative algebra and 
that the functor $I_\C$ has a structure of tensor functor.

It follows from definitions (see \cite{EO, O1}) that the multiplication on the algebra $A$ can
be described as follows.
Let $\mu \in \Hom(F_\C(I_\C(\be)),\be)$ be the image of $\id$ under the canonical isomorphism
$\Hom(I_\C(\be),I_\C(\be))\simeq \Hom(F_\C(I_\C(\be)),\be)$. The the multiplication morphism 
$m: A\ot A\to A$ is the image
of $\mu \ot \mu$ under the isomorphism $\Hom(F_\C(I_\C(\be))\ot F_\C(I_\C(\be)),\be)\simeq 
\Hom(I_\C(\be)\ot I_\C(\be),I_\C(\be))$. 
By definition the commutativity of $A$ 
means that $m\in \Hom(I_\C(\be)\ot I_\C(\be),I_\C(\be))$ is invariant under the action of the braiding
permuting two copies of $I_\C(\be)$. Using the definition of $m$
we see that this is equivalent to the invariance of 
$\mu \ot \mu \in \Hom(F_\C(I_\C(\be))\ot F_\C(I_\C(\be)),\be)$ under the braiding $c_{A,A}$ permuting 
the two copies of $F_\C(I_\C(\be))\in \C$ (note that $F_\C(I_\C(\be))$ has a canonical
lift to $\Z(\C)$, namely $A=I_\C(\be)$, so we can talk about the braiding). The naturality
of the braiding with a central object implies the commutativity of the following diagram:
\begin{equation}  \label{Acomm}
\xymatrix{F_\C(I_\C(\be))\ot F_\C(I_\C(\be))\ar[rr]^{c_{A,A}} \ar[d]_{\id \ot \mu}&&
F_\C(I_\C(\be))\ot F_\C(I_\C(\be))\ar[d]^{\mu \ot \id}\\ 
F_\C(I_\C(\be))\ot \be \ar[r]^{\id}\ar[dr]_{\mu \ot \id}&F_\C(I_\C(\be))\ar[r]^{\id} &
\be \ot F_\C(I_\C(\be))\ar[dl]^{\id \ot \mu}\\
&\be \ot \be&}
\end{equation}
Applying the functor $\Hom(?,\be)$ to this diagram we obtain the desired invariance of $\mu \ot \mu$.

For any $X\in \C$ let $\mu_X: F_\C(I_\C(X))\to X$ be the image of $\id$ under the canonical
isomorphism $\Hom(I_\C(X), I_\C(X))\simeq \Hom(F_\C(I_\C(X)), X)$ (so we have $\mu_\be =\mu$
in the notation used above) and for $X,Y\in \C$
let $\mu_{X,Y}: I_\C(X)\ot I_\C(Y)\to I_\C(X\ot Y)$ be the image of $\mu_X\ot \mu_Y$ under the
canonical isomorphism $\Hom(F_\C(I_\C(X))\ot F_\C(I_\C(Y)), X\ot Y)\simeq 
\Hom(I_\C(X)\ot I_\C(Y), I_\C(X\ot Y))$ (in the notation above $\mu_{\be,\be}=m$ is the
multiplication morphism on $A=I_\C(\be)$ and $\mu_{X,\be}$ is the morphism making $I_\C(X)$ into
right $A-$module). It is straightforward to verify that $\mu_{X,Y}$ satisfies
all the axioms of a tensor functor except for being an isomorphism. In particular, 
the morphism $\mu_{\be, X}$ makes $I_\C(X)$ into left $A-$module; moreover $\mu_{\be,X}$
and $\mu_{X,\be}$ make $I_\C(X)$ into $A-$bimodule. The diagram similar to \eqref{Acomm}
shows that $\mu_{\be,X}$ can be described as a composition 
$A\ot I_\C(X)\xrightarrow{c_{A,I_\C(X)}}I_C(X)\ot A\xrightarrow{\mu_{X,\be}}I_\C(X)$,
so the structure of $I_\C(X)$ as $A-$bimodule is the same as the structure used in the
definition of tensor structure on $A-\mbox{mod}_{\Z(\C)}$.

It is immediate to check that $\mu_{X,Y}$ factorizes through the canonical map 
$I_\C(X)\ot I_\C(Y)\to I_\C(X)$ as
$I_\C(X)\ot I_\C(Y)\to I_\C(X)\ot_AI_\C(Y)\xrightarrow{\tilde \mu_{X,Y}} I_\C(X\ot Y)$
and that $\tilde \mu_{X,Y}$ satisfies all the axioms of a tensor functor with a possible exception 
of being an isomorphism. Finally one verifies that for $X=F_\C(Z)$ with $Z\in \Z(\C)$ we have
$I_\C(X)=Z\ot A$ (as $A-$modules) and under this isomorphism $\tilde \mu_{X,Y}$ goes
to the canonical isomorphism $I_\C(X)\ot_AI_\C(Y)=Z\ot I_\C(Y)\simeq I_\C(F_\C(Z)\ot Y)$
from \cite[Proposition 3.39 (iii)]{EO}. Since the functor $F_\C$ is surjective (see
\cite[Proposition 3.39 (i)]{EO}) we get that $\tilde \mu_{X,Y}$ is always an isomorphism.
Thus the isomorphisms $\tilde \mu_{X,Y}$ define a tensor structure on the functor $I_\C$
and Lemma is proved.
\end{proof}

Let $\C,\D$ be fusion categories such that there is a braided tensor equivalence
$a: \Z(\C)\cong  \Z(\D)$.
Since $I_\D(\be)$ is a commutative algebra in $\Z(\D)$ and $a$ is a {\em braided}
equivalence, we have that
$L:=a^{-1}(I_\D(\be))$  is a commutative  algebra in $\Z(\C)$. 
Furthermore, 
\begin{equation}
\label{D=L-mod}
\D  \cong L-\mbox{mod}_{\Z(\C)}
\end{equation}
as a fusion category.

Note that $L$  is indecomposable in $\Z(\C)$ but 
might be decomposable as an algebra {\em in $\C$}, i.e., 
\[
L =\bigoplus_{i\in J} \, L_i,
\]
where $L_i,\, i\in J,$ are indecomposable algebras in $\C$ such that the 
multiplication of $L$ is zero on $L_i \ot L_j,\, i\neq j$ 
(e.g., if $\C =\Rep(G)$ then $L =\Fun(G,\, k)$ with the adjoint action of $G$
and $J$ is the set of conjugacy classes of $G$).
\footnote{Here and below 
we abuse notation and write $L$ for an object of $\Z(\C)$ and its
forgetful image in $\C$.}

We would like to show  that for each $i\in J$
\begin{equation}
\label{Lbimod}
L_i-\mbox{bimod}_\C \cong L-\mbox{mod}_{\Z(\C)}.
\end{equation}
In view of \eqref{D=L-mod} this will  mean that $\D$ is dual to $\C$ 
with respect to the $\C$-module category  $L_i-\mbox{mod}_\C$ 
for any $i\in J$.

Consider the following commutative diagram of tensor functors:

\begin{equation*}
\xymatrix{
\Z(\C)  \ar[d]_{Z \mapsto Z\ot L} \ar[rr]^{Z\mapsto Z\ot L_i} & &
\Z(L_i-\mbox{bimod}_\C) \ar[d]^{F_{L_i-\mbox{bimod}_\C}} \\
L-\mbox{mod}_{\Z(\C)}  \ar[r]^(.35){F_\C} & \bigoplus\, L_i-\mbox{bimod}_\C \subset L-\mbox{bimod}_\C
\ar[r]^(.63){\pi_i} & L_i-\mbox{bimod}_\C.
}
\end{equation*}
 Here $\pi_i$ is a projection from 
$L-\mbox{bimod}_\C =\oplus_{ij}\, (L_i-L_j)-\mbox{bimod}_\C$ to its
$(i,\,i)$ component. We have $\pi_i(X\ot L) = X \ot L_i$
for all $X\in \C$.  The top arrow is an equivalence by \cite[Remark 3.18]{M2} 
(see also \cite[Corollary 3.35]{EO})  and the forgetful functor
$\Z(L_i-\mbox{bimod}_\C)\to L_i-\mbox{bimod}_\C$ (the right down arrow)
is surjective. Hence,  the composition  $F_i := \pi_i F_\C$ of the functors
in the bottom row is surjective. But $F_i$ is a tensor functor between
fusion categories of equal Frobenius-Perron dimension and hence it is an 
equivalence by \cite[Proposition 2.20]{EO}.
\end{proof}

\begin{remark}
\begin{enumerate}
\item[(1)] The above characterization of Morita equivalence was announced earlier and independently
by A.~Kitaev and M.~M\"uger.
\item[(2)] For group-theoretical categories Theorem~\ref{Morita eq iff} was proved in \cite{NN}.
\item[(3)] A crucial idea of the proof of Theorem~\ref{Morita eq iff} (which is to consider a commutative
algebra $L\in \Z(\C)$ as an algebra in $\C$) also appears in \cite[Theorem 3.22]{KR}. 
\end{enumerate}
\end{remark}

\begin{lemma}
\label{mor} 
Let $G$ be a finite group, let $\D$ be a $G$-extension of a fusion category $\D_0$,
and let $\tilde{\D}_0$  be a fusion category Morita equivalent to $\D_0$.
There exists a $G$-extension  $\tilde{\D}$ of $\tilde{\D}_0$
which is  Morita equivalent to $\D$ .
\end{lemma}
\begin{proof} Let $A$ be an indecomposable algebra in $\D_0$ such that $\tilde{\D}_0$ is equivalent
to the category of $A$-bimodules in $\D_0$.  Observe that the tensor 
category $\tilde{\D}=A-\mbox{bimod}_\D$ of $A$-bimodules {\em in} $\D$ 
inherits the  $G$-grading  (since $A$ belongs to the trivial component of $\D$). 
Since the category of $A$-modules in $\D$ is indecomposable,  $\tilde{\D}$ 
is a fusion category.
\end{proof}


We are now ready to prove Theorem~\ref{Th0}.

Suppose $\C$ is Morita equivalent to a $G$-extension
$\widetilde{\C}$ of $\D$. By  \cite[Remark 3.18]{M2} 
there is a braided tensor equivalence $\Z(\C)\cong
\Z(\widetilde{\C})$. 
By Proposition~\ref{eqext1}(ii)
$\Z(\widetilde{\C})$ contains a Tannakian subcategory 
$\E=\Rep(G)$ with the specified property, as desired. 

Conversely, suppose that $\Z(\C)$ contains a Tannakian subcategory $\E=\Rep(G)$ such that
the de-equivariantization $\Z$ of $\E'$ by $\E$ is equivalent to $\Z(\D)$ as a braided tensor category.
Let $I: \D \to \Z(\D) \cong \Z$ be the composition of the left adjoint of the forgetful 
functor $\Z(\D) \to \D$ with the equivalence $\Z(\D) \cong \Z$.
Then $A_1 := I(\be)$ is a commutative algebra in $\Z$. 
Let $J:  \Z = A_1-\mbox{mod}_{\E'} \to \E'$ be the functor forgetting
the $A_1$-module structure then $A := J(A_1)$
is a commutative algebra in $\E'$ and, hence, in $\Z(\C)$.

It was explained in \cite{DGNO2}  that for every $Z\in \Z(\C)$  the object $Z \ot A$ has
a structure of an object in the center of $A-\mbox{mod}_{\C}$ and that the
functor $\Z(\C) \to \Z(A-\mbox{mod}_{\C}) : Z \mapsto A \ot Z$
is a braided tensor equivalence. By Theorem~\ref{Morita eq iff} $\C$ and $A-\mbox{mod}_{\C}$
are Morita equivalent. 

The composition $\Z(\C) \cong \Z(A-\mbox{mod}_{\C})\to A-\mbox{mod}_{\C}$
identifies with the free $A$-module functor.
This functor takes $\E=\Rep(G)\subset \Z(A-\mbox{mod}_{\C})$ to $\Vec$. By Proposition~\ref{eqext1}
$A-\mbox{mod}_{\C}$ is a $G$-extension of some fusion category  $\tilde{\D}$ and 
there is a braided tensor  equivalence $\Z(\D) \cong \Z(\tilde{\D})$.
By Theorem~\ref{Morita eq iff} $\D$ and $\tilde{\D}$ are Morita equivalent.
So $\D$ is Morita equivalent to a $G$-extension of $\D$ by Lemma~\ref{mor},
as required.

\section{Properties of weakly group-theoretical 
and solvable fusion categories}
\label{Sect 4}

\subsection{Properties of weakly group-theoretical categories}

The basic properties of weakly group-theoretical fusion categories
(see Definition \ref{wgt}) are summarized in the following two Propositions. 

\begin{proposition}
\label{basicp1-bis}
The class of weakly group-theoretical categories is closed
under taking extensions, equivariantizations, Morita
equivalent categories, tensor products, 
the  center, subcategories and component categories of 
quotient categories.
\end{proposition}
\begin{proof}
The invariance under taking 
Morita equivalent categories and tensor products is obvious.
The invariance under taking extensions follows from Lemma~\ref{mor} and
the invariance under equivariantizations follows
from Proposition~\ref{equiv}.
The invariance under taking the  center 
then follows from Morita invariance, as $\Z(\C)$ is Morita
equivalent to $\C\boxtimes \C^{op}$.
The rest of the proof is similar to the proof of Proposition 8.44
in \cite{ENO}. Namely, to prove the invariance under 
taking subcategories, let $\C$ be a weakly group-theoretical
category, and $\D\subset \C$ a fusion subcategory. Let $\M$ be an
indecomposable $\C$-module category such that $\C_\M^*$ is
nilpotent. Then every component category of $\D_\M^*$ is
nilpotent, since it is easy to see that every component category
in a quotient of a nilpotent category is nilpotent. 
The case of a component in a quotient category reduces to the
case of a subcategory by taking duals.  
\end{proof}

\begin{proposition}
\label{basicp1}
A fusion category $\C$ is weakly group-theoretical if and only if 
there exists a sequence of non-degenerate braided categories 
\[
\Vec=\D_0,\, \D_1,\dots, \D_n=\Z(\C)
\]
and finite groups $G_1,\dots ,G_n$, such that for all $1\le i\le n$, the Tannakian category
$\Rep(G_i)$ is contained in $\D_i$ as an isotropic subcategory, and 
$\D_{i-1}$ is the de-equivariantization of 
the M\"uger centralizer $\Rep(G_i)'$ of $\Rep(G_i)$ in $\D_i$. 
\end{proposition}
\begin{proof} 
To proof the ``only if'' part it suffices to assume that $\C$ is nilpotent.
Suppose first that $\C$ is a $G$-extension 
of another category $\D$. By Proposition~\ref{eqext1} $\Z(\C)$
contains a Tannakian subcategory $\Rep(G)$ such that the 
de-equivariantization of 
$\Rep(G)'$ by $\Rep(G)$ is $\Z(\D)$. Since every nilpotent category is obtained from $\Vec$
by a sequence of extensions, this implies the desired statement. 

To prove the ``if'' part, we argue by induction in $n$.
For $n=1$ we must have a Lagrangian subcategory $\Rep(G_1)\subset \Z(\C)$
and so $\C$ is group-theoretical by \cite[Corollary 4.14]{DGNO1}. Suppose the statement
is true for $n=l$ and let $\Vec=\D_0$, $\D_1$,...,$\D_{l+1}=\Z(\C)$
be a sequence as in the statement of the Proposition.  By Theorem~\ref{Mor eq  graded}
$\C$ is  Morita equivalent to a $G_1$-extension of a fusion category 
$\tilde{\C}$ such that $\Z(\tilde{\C})$ is braided equivalent to 
the de-equivariantization of $\Rep(G_1)'$ by $\Rep(G_1)$ in $\D_1$.
By induction, $\tilde{\C}$ is weakly group-theoretical, hence
$\C$ is weakly group-theoretical by Proposition~\ref{basicp1-bis}.
\end{proof}

\begin{remark}
Note that the class of group-theoretical categories 
is {\em not} closed under taking extensions and equivariantizations, 
see \cite{Nk2} and \cite[Remark 8.48]{ENO}. 
\end{remark} 

\subsection{Properties of solvable fusion categories.}

Let $\C$ be a fusion category.

\begin{proposition}\label{basicp2}
The following conditions are equivalent: 
\begin{enumerate}
\item[(1)] $\C$ is solvable in the sense of Definition~\ref{solv}(i).
\item[(2)] $\Z(\C)$ admits a
chain as in Proposition \ref{basicp1}, where all the groups
$G_i$ are cyclic of prime order. 
\item[(3)] There is a sequence of fusion categories 
$\C_0=\Vec,\, \C_1,\dots, \C_n = \C$
and a sequence $G_1, \dots G_n$ of cyclic groups of prime order 
such that $\C_{i}$ is obtained from $\C_{i-1}$ either by a 
$G_i$-equivariantization or as a $G_i$-extension. 
\end{enumerate}
\end{proposition}
\begin{proof} 
(1) $\Rightarrow$ (2). This follows by iterating  Proposition~\ref{eqext1}.

(2) $\Rightarrow$ (3). We argue by induction in $n$. 
Consider the image of $\Rep(G_n)$ in $\C$ under the forgetful functor $\Z(\C)\to \C$. 
Since $G_n$ is cyclic of prime order, either $\Rep(G_n)$
maps to $\Vec$, in which case $\C$ is a $G_n$-extension of some
category $\D$ by Proposition~\ref{eqext1}(i), or $\Rep(G_n)$ embeds into $\C$, 
in which case $\C$ is a $G_n$-equivariantization of some category $\D$
by Proposition~\ref{eqext2}(i). 
In both cases, $\Z(\D)=\D_{n-1}$, and by the induction assumption
$\D$ satisfies (3), so we are done.

(3) $\Rightarrow$ (1). By Proposition~\ref{equiv} $\C_i$ is Morita
equivalent to a $G_i$-extension of $\C_{i-1},\, i=1,\dots, n$.
Combining induction with Lemma~\ref{mor} we see that $\C$ is 
Morita equivalent to a cyclically nilpotent fusion category, i.e., $\C$ is solvable.
\end{proof}

\begin{proposition}
\label{basicp2-bis}
\begin{enumerate}
\item[(i)] The class of solvable categories is closed
under taking extensions and equivariantizations by solvable
groups, Morita equivalent categories, tensor products, 
 center, subcategories and component categories of 
quotient categories. 
\item[(ii)]
The categories $Vec_{G,\omega}$ and $\Rep(G)$ are solvable
if and only if $G$ is a solvable group.
\item[(iii)]
A braided nilpotent fusion category is solvable.
\item[(iv)] 
A solvable fusion category $\C\ne \Vec$ contains a nontrivial
invertible object. 
\end{enumerate}
\end{proposition}
\begin{proof}
(i) As in the proof of Proposition \ref{basicp1-bis},
everything follows from the easy fact that a component category in a
quotient of a cyclically nilpotent category is cyclically
nilpotent. 

(ii) One direction is obvious, since if $G$ is solvable,
$\Vec_{G,\omega}$ is cyclically nilpotent.  Since $\Rep(G)$
is Morita equivalent to $\Vec_G$ it is also solvable by (i).

To prove the converse implication it suffices to show that
if $\Rep(G)$ is solvable then so is $G$.  Indeed, $\Z(\Vec_{G,\omega})$
contains $\Rep(G)$ as a fusion subcategory, so the solvability of 
$\Vec_{G,\omega}$ implies solvability of $\Rep(G)$ by (i).
We have two possibilities: either $\Rep(G)$ is an $H$-extension
or $\Rep(G) =\C^H$ for some fusion category $\C$, 
where $H$ is a cyclic group of prime order.
In the former situation $G$ must have a non-trivial center $Z$
and we can pass to the fusion subcategory $\Rep(G/Z)\subset \Rep(G)$
which is again solvable by (i). In the latter situation $\Rep(G)$
contains a fusion subcategory of prime order by Proposition~\ref{eqext1}(i),
therefore, $G$ contains a normal subgroup $G_1$ of prime index
and we can pass to the solvable  quotient category $\Rep(G_1)$.
So the required statement follows by induction.


(iii) Follows from \cite[Theorem 6.12]{DGNO1} combined with \cite[Theorem 8.28]{ENO}.   

(iv) The proof is by induction in the dimension of $\C$. 
The base of induction is clear, and only the induction step needs to be
justified. If $\C$ is an extension of a smaller solvable category $\D$, then 
either $\D\ne \Vec$ and the statement follows from the induction
assumption, or $\D=\Vec$ and $\C$ is pointed, so the statement
is obvious. On the other hand, 
if $\C$ is a $\Bbb Z/p$-equivariantization of a smaller
solvable category $\D$, then $\Rep(\Bbb Z/p)$ sits inside 
$\C$, so we are done.   
\end{proof}

\begin{remark} 
\label{MW solvable}
\begin{enumerate}
\item[(1)] Note that a non-braided nilpotent fusion category need not be solvable
(e.g., $\Vec_G$ for a non-solvable group $G$).
\item[(2)] The notion of a solvable fusion category is close in
spirit to the notions of upper and lower solvable 
and semisolvable Hopf algebras introduced by 
Montgomery and Witherspoon \cite{MW}. However, 
we would like to note that a semisimple Hopf 
algebra $H$ such that $\Rep(H)$ is  
solvable in our sense is not necessarily 
upper or lower semisolvable in the sense of \cite{MW}. 
For example, Galindo and Natale constructed in \cite{GN}
self-dual Hopf algebras without nontrivial normal Hopf subalgebras 
as twisting deformations of solvable groups. 
Clearly, the representation category of any such Hopf algebra is solvable.
It is also easy to construct an example of an upper and lower
solvable semisimple Hopf
algebra $H$, such that $\Rep(H)$ is not solvable.
For this, it suffices to take the Kac algebra associated to the
exact factorization of groups $A_5=A_4\cdot \Bbb Z/5\Bbb Z$.
\end{enumerate}
\end{remark}

\section{Module categories over equivariantized  categories}
\label{Sect 5}

Let $\C$ be a fusion category and let $G$ be a finite group
acting on $\C$. In this Section we obtain a description of module
categories over the equivariantization $\C^G$.

Let $\M$ be a $\C$-module category, and let $t$ be a tensor
autoequivalence of $\C$.  Define a twisted $\C$-module category
$\M^t$ by setting $\M^t =\M$ as an abelian category and defining
a new action of $\C$:
\[
X \ot^t M :=  t(X) \ot M,
\]
for all objects $M$ in $\M$ and $X$ in $\C$, cf.\ \cite{Nk2}.
Given  a  $\C$-module functor $F: \M \to \N$,  we define a $\C$-module functor
$F^t: \M^t \to \N^t$ in an obvious way. Given a natural transformation $\nu: F\to G$ 
between  $\C$-module functors $F,\, G: \M\to \N$  we define a natural
transformation $\nu^t: F^t \to G^t$.

\begin{remark}
If $A$ is an algebra in $\C$ such that $\M$ is equivalent to the category of 
$A$-modules in $\C$ then $\M^t$ is equivalent to the category of
$t(A)$-modules in $\C$.
\end{remark}

An action of a group $G$ on $\C$ gives rise to $\C$-module 
equivalences 
\[
\Gamma_{g,h}:  (\M^h)^g \cong \M^{gh},\, g, h\in G.
\]

\begin{definition}
A {\em $G$-equivariant $\C$-module category} is a pair consisting
of a $\C$-module category $\M$ along with $\C$-module
equivalences $U_g : \M^g \xrightarrow{\sim} \M,\, g\in G,$ and
natural isomorphisms of tensor functors $\mu_{g,h} : U_{gh}\, \Gamma_{g, h}
\xrightarrow{\sim} U_g\, (U_h)^g,\, g,h\in G,$
satisfying the following  compatibility conditions:
\begin{equation}
\label{compcond}
 (\mu _{f,g} \, (U_h)^{fg} ) \circ( \mu_{fg,h}\, \Gamma_{f,g} )= 
  (U_f \,(\mu_{g,h})^f) \circ (\mu_{f,gh}\, \Gamma_{g,h}),
  \qquad f,\,g,\, h\in G.
\end{equation}
\end{definition}

\begin{definition}
Let $\M$ be a $G$-equivariant $\C$-module category. An {\em equivariant object}
in $\M$ is a pair consisting of an object $M$ of $\M$ along with isomorphisms
$v_g : U_g(M) \xrightarrow{\sim} M,\, g\in G,$ such that the diagrams 
\begin{equation*}
\label{equivariantX}
\xymatrix{U_g(U_h(M))\ar[rr]^{U_g(v_h)} \ar[d]_{\mu_{g,h}(M)
}&&U_g(M)\ar[d]^{v_g}\\ U_{gh}(M)\ar[rr]^{v_{gh}}&&M}
\end{equation*}
commute for all $g,h\in G$.
\end{definition}

Let $H$ be a subgroup of $G$ and let $\M$ be an $H$-equivariant
$\C$-module category.  Let $\M^H$ denote the category of
equivariant objects in $\M$.  Then $\M^H$ is a $\C^G$-module
category. Namely, the equivariant structure on $X\ot M$, where
$X$ is an object of $\C^G$ and $M$ is an object of $\M^H$, is
given by the product of equivariant structures of $X$ and $M$.

\begin{proposition}
\label{CG-modules}
Every indecomposable $\C^G$-module category is equivalent to 
one of the form $\M^H$, where $H$ is a subgroup of $G$, and $\M$
is an $H$-equivariant indecomposable $\C$-module category.
\end{proposition}

\begin{proof}
Consider the crossed product category $\C \rtimes G$
(see Subsection \ref{cro}). Indecomposable $(\C \rtimes G)$-module categories
were studied in \cite{Nk2}. Every such module category $\N$
decomposes into a direct sum of $\C$-module categories $\N
=\oplus_{s\in S}\, \N_s$, where $S$ is a homogeneous $G$-set.
Let $H$ be the stabilizer of $s\in S$ so that $S\cong G/H$. It
follows that $\M:= \N_s$ is an $H$-equivariant $\C$-module
category which completely determines $\N$.  By Proposition
\ref{equiv}, any indecomposable $\C^G$-module
category is equivalent to the category of $(\C \rtimes G)$-module
functors from $\C$ to $\N$ for some $\N$ as above. 
It is easy to see that such functors
correspond to equivariant objects in $\M$.
\end{proof}

\begin{example}
Let $\C=\Vec$.  We have $\Vec\rtimes G = \Vec_G$ and $\Vec^G =\Rep(G)$.
An $H$-equivariant $\Vec$-module category is nothing but a $2$-cocycle
$\mu \in Z^2(H, \Bbb C^*)$.  An equivariant object in this category is the 
same thing as a projective representation of $H$ with the Schur
multiplier $\mu$. Thus, our description agrees with that of \cite{O2}.
\end{example}

\section{Proof of Theorem \ref{Th1}}
\label{Sect 6}

Let $\C$ be a fusion category, and let $\M$ be an indecomposable 
module category over $\C$. We will denote the Frobenius-Perron 
dimension of $X\in \M$ normalized as in Definition~\ref{strong Frobenius}
by  $\FPdim_\M(X)$.  

\begin{definition} Let $m\ne 0$ be an algebraic integer. 
Let us say that a fusion category $\C$ has the {\em strong
$m$-Frobenius property} if for any
indecomposable $\C$-module category $\M$ and any simple
$X\in \M$, the ratio $\frac{\FPdim(\C)}{m\FPdim_\M(X)}$
is an algebraic integer. 
\end{definition}

\begin{proposition}\label{extequi}
Let $\C$ be a fusion category having the strong $m$-Frobenius property, 
and let $G$ be a finite group. Then 
\begin{enumerate}
\item[(1)] A $G$-equivariantization of $\C$ has the  strong
$m$-Frobenius property, and 
\item[(2)] A $G$-extension of $\C$ has the  strong $m\sqrt{|G|}$-
Frobenius property.
\end{enumerate}
\end{proposition}

\begin{proof}
(1)  By Proposition~\ref{CG-modules}
every indecomposable $\C^G$-module category $\M$ 
is equivalent to $\N^H$, where $H$ is a subgroup of $G$ and $\N$ is
an $H$-equivariant indecomposable $\C$-module category.
For any simple object $X$ in $\M=\N^H$
choose and fix its simple constituent $Y$ in $\N$.
Let $\mbox{Stab}(Y)$ denote the stabilizer of $Y$ in $H$.  Then $X$ corresponds to an
irreducible representation $\pi$ of $\mbox{Stab}(Y)$ and 
\begin{equation}
\label{X vs Y}
\FPdim_\N (X) =  \deg(\pi) [H : \mbox{Stab}(Y)] \FPdim_\N(Y). 
\end{equation}
We claim that 
\begin{equation}
\label{Pasha asked to explain}
\frac{\FPdim_\M(X)} {\FPdim_\N(X)} =\sqrt{[G:H]}.
\end{equation}
Indeed, we have
\begin{eqnarray*}
\sum_{X}\, \FPdim_\N(X)^2
&=& \sum_Y\, \left( \sum_\pi\, \deg(\pi)^2 \right) [H : \mbox{Stab}(Y)]^2 \FPdim_\N(Y)^2 \\
&=& |H| \sum_Y\,  [H : \mbox{Stab}(Y)] \FPdim_\N(Y)^2 \\
&=& |H| \FPdim(\C).
\end{eqnarray*}
On the other hand, $\sum_{X}\, \FPdim_\M(X)^2 = \FPdim(\C^G) =|G|\FPdim(\C)$.
Combining these two equations we  obtain \eqref{Pasha asked to explain}.
Comparing with \eqref{X vs Y} we see that
\[
\frac{\FPdim(\C^G)}{\FPdim_\M(X)} 
= \frac{\sqrt{[G:H]}\, |\mbox{Stab}(Y)|}{\deg(\pi)} 
\times  \frac{\FPdim(\C)}{\FPdim_\N (Y)},
\]
and so $\C^G$ has the strong $m$-Frobenius property.

(2)  Let $\D =\oplus_{g\in G}\, \D_g,\, \D_e =\C,$ 
be a $G$-extension of $\C$, and let
$\M$ be an indecomposable $\D$-module category. 
Let $\M =\oplus_{s\in S}\, \M_s$
be its decomposition as a $\C$-module category.  
It was shown in \cite{GNk} that $S$
is a homogeneous $G$-set and $\FPdim(\M_s)/\FPdim(\M_t) =1$ 
for all $s,t\in S$. 
Let $H\subset G$ be a subgroup such that  $S =G/H$. Then 
for any simple object $X $ in $\M_s$
we have   $\tfrac{\FPdim_\M(X)}{\FPdim_{\M_s}(X)} =\sqrt{|H|}$, and therefore,
\[
\frac{\FPdim(\D)}{\FPdim_\M(X)}  
= \frac{|G|}{\sqrt{|H|}} 
\times \frac{\FPdim(\C)}{\FPdim_{\M_s}(X) } ,
\]
and so $\D$ has the strong $m\sqrt{|G|}$-Frobenius property.
\end{proof}

\begin{lemma}
\label{G-crossed}
Let  $\D$ be a non-degenerate braided fusion category containing a Tannakian
subcategory $\E=\Rep(G)$. Let $\Z$ be the de-equivariantization of $\E'$ by $\E$.
Then $\D$ is equivalent, as a fusion category, to a $G$-equivariantization
of a $G$-extension of~$\Z$. 
\end{lemma}
\begin{proof}
Let $A=\mbox{Fun}(G)$ be the algebra of functions on $G$. The category $\D_G$ 
of $A$-modules in $\C$ is faithfully $G$-graded with the trivial component $\Z$,
and $\D$ is a $G$-equivariantization of  $\D_G$, see \cite{K, M3}.
\end{proof}

\begin{corollary}\label{exttfrob}
Let $\C$ be a weakly group-theoretical fusion category. Then its
center $\Z(\C)$ has the strong $\sqrt{\FPdim(\C)}$-Frobenius property. 
\end{corollary}

\begin{proof}
It follows from Proposition~\ref{basicp1} and   Lemma~\ref{G-crossed}  that there exists a
sequence of finite groups $G_1,\dots,G_n$ such that $\Z(\C)$ can be
obtained from $\Vec$ by the following sequence of $2n$ operations: 
$G_1$-extension, $G_1$-equivariantization,...,$G_n$-extension,
$G_n$-equivariantiza\-tion. Therefore, the result follows from 
Proposition \ref{extequi}.  
\end{proof} 

Now we are in a position to prove Theorem \ref{Th1}. 
Let $\C$ be a weakly group-theoretical fusion category, and let $\M$ be 
an indecomposable module category over $\C$. Let $\widetilde{\M}$ be the 
pullback of $\M$ under the forgetful functor $\Z(\C)\to \C$. 
Then it is obvious that for any $X\in \M$, one has 
$\FPdim_{\widetilde{\M}}(X)=\FPdim_{\M}(X)\sqrt{\FPdim(\C)}$. 
On the other hand, Corollary \ref{exttfrob} implies that 
$\FPdim(\C)^2/\FPdim_{\widetilde{\M}}(X)$ is an algebraic integer divisible by
$\sqrt{\FPdim(\C)}$. This implies that $\FPdim(\C)/\FPdim_\M(X)$ 
is an algebraic integer, i.e., $\C$ has the strong Frobenius
property.

\section{Nondegenerate and slightly degenerate categories
with a simple object of prime power dimension}
\label{Sect 7}

In this section we will prove several results on 
non-degenerate and slightly degenerate braided categories containing a simple object
of prime power dimension. These results will be of central
importance for the proof of Theorem \ref{Th2} and further results
of the paper, and are parallel to the character-theoretic lemmas
used in the classical proof of Burnside's theorem in
group theory. 

\begin{lemma}\label{key} Let $X$ and $Y$ be two simple objects of
an integral braided category 
with coprime dimensions\footnote{Here and below, we use shortened
notation $d_X$ for the Frobenius-Perron dimension $\FPdim(X)$.} $d_X,d_Y$. Then one of two possibilities hold:
\begin{enumerate}
\item[(i)] $X$ and $Y$ projectively centralize each other (i.e. the
square of the braiding on $X\ot Y$ is
a scalar);
\item[(ii)] $s_{X,Y}=0$ (where $s$ is the S-matrix).
\end{enumerate}
\end{lemma}

\begin{proof} It suffices to consider the case when the category
is non-degenerate, since any braided category can be embedded into a
non-degenerate one (its  center). In this case,  
by the Verlinde formula, $\frac{s_{X,Y}}{d_X}$ and $\frac{s_{X,Y}}{d_Y}$ are algebraic
integers. Since $d_X$ and $d_Y$ are coprime, $\frac{s_{X,Y}}{d_Xd_Y}$ is also an
algebraic integer. Since $s_{X,Y}$ is a sum of $d_Xd_Y$ roots of unity, we see that 
$\frac{s_{X,Y}}{d_Xd_Y}$ is either a root of unity (in which case
the square of the braiding must be a scalar, option (i)), or 0
(option (ii)).
\end{proof}

\begin{corollary}\label{primpow} Let $\E$ be an integral non-degenerate braided category which contains a  simple object $X$ with
dimension $d_X=p^r$, $r>0$, where $p$ is a prime. Then $\E$
contains a nontrivial symmetric subcategory.
\end{corollary}

\begin{proof} We first show that $\E$ contains a nontrivial
proper subcategory. Assume not.
Take any simple $Y\ne \be$ with $d_Y$ coprime to $d_X$. We claim that $s_{X,Y}=0$. Indeed,
otherwise $X$ and $Y$ projectively centralize each other, hence $Y$ centralizes $X\ot X^*$,
so the M\"uger centralizer of the category generated by $Y$ is
nontrivial, and we get a nontrivial proper subcategory, a  contradiction. 

Now let us use the orthogonality of columns $(s_{X,Y})$ and $(d_Y)$
of the $S$-matrix: 
$$\sum_{Y\in {\rm Irr}\E}\frac{s_{X,Y}}{d_X}d_Y=0.
$$
As we have shown, all the nonzero summands in this sum, except
the one for $Y=\be$, come from objects $Y$ of dimension divisible
by $p$. Therefore, all the 
summands in this sum except for the one for $Y=\be$ (which equals
$1$) are divisible by $p$. This is a contradiction.

Now we prove the corollary by induction in $\FPdim(\E)$. 
Let $\D$ be a nontrivial proper subcategory of $\E$. 
If $\D$ is degenerate, then its M\"uger center is a nontrivial
proper symmetric subcategory of $\E$, so we are done. 
Otherwise, $\D$ is non-degenerate, and by Theorem \ref{MuTh},
$\E=\D\boxtimes \D'$. Thus $X=X_1\otimes X_2$, where $X_1\in \D$,
$X_2\in \D'$ are simple. Since the dimension of $X_1$ or $X_2$ is
a positive power of $p$, we get the desired statement from the
induction assumption applied to $\D$ or $\D'$ (which are non-degenerate braided
categories of smaller dimension). 
\end{proof}

\begin{remark} The proof above shows that for a simple object $X$ of an integral non-degenerate
braided category $\E$ with prime power dimension we can find another nontrivial simple object $Y$
such that $X$ and $Y$ projectively centralize each other. This can be used to give a proof
of Burnside theorem that a finite group $G$ with a conjugacy class $C$ of prime power size 
can not be simple (together with Sylow theorem this implies immediately the solvability
of groups of order $p^aq^b$) as follows. Assume that $G$ is simple (it is also nonabelian
since it contains conjugacy class $C$ of size $>1$). Then $G$ is generated by any of its
nontrivial conjugacy classes and $\Rep(G)$ has no nontrivial fusion subcategories. This implies
that the category $\Z(\Rep(G))$ has a unique proper  fusion subcategory, namely $\Rep(G)$
(we recall that the category $\Z(\Rep(G))$ is identified with the category of sheaves on $G$ which
are $G-$equivariant with respect to the adjoint action). Now let $X$ be the simple object of
$\Z(\Rep(G))$ which corresponds to a trivial sheaf supported on $C$; then $d_X=|C|$ is a prime
power, hence $X$ projectively centralizes some nontrivial object $Y$. The object $Y$ is not
invertible since the group $G$ is simple nonabelian; hence $X$ generates a nontrivial 
fusion subcategory of $\Z(\Rep(G))$ contained in M\"uger centralizer of $Y\ot Y^*$. This
is a contradiction since $X$ is not contained in $\Rep(G)\subset \Z(\Rep(G))$ (recall
that the subcategory $\Rep(G)$ consists of objects supported on the unit element $e\in G$).
\end{remark}

\begin{proposition}\label{almnontan}
Let $\C$ be a slightly degenerate integral braided category,
which contains a simple object $X$ of dimension $p^r$ for some 
prime $p>2$. Then $\C$ contains a nontrivial Tannakian subcategory. 
\end{proposition}

\begin{proof} The proof is by induction on the dimension of $\C$. Let $\B$ 
be the category spanned by the invertible objects of $\C$. 
Then the M\"uger center of $\B$ contains the category ${\rm SuperVec}$.

If the M\"uger center of $\B$ is bigger than ${\rm SuperVec}$, then it
contains a nontrivial Tannakian subcategory, and we are done. 
Otherwise, by Proposition \ref{almnon}(ii), $\B={\rm SuperVec}\boxtimes \B_0$, 
where $\B_0$ is a pointed non-degenerate braided category. 
If $\B_0$ is nontrivial, then $\C=\B_0\boxtimes
\B_0'$, and $\B_0'$ is slightly degenerate, so 
we are done by the induction assumption. Thus, it suffices to
consider the case $\B={\rm SuperVec}$, which we do from now on.  

Let $\be$ and $\chi$ 
be the simple objects of ${\rm SuperVec}\subset \C$
(which are the only invertible objects of $\C$). 
Let $Y$ be a non-invertible simple object of $\C$ of dimension
not divisible by $p$. 

Assume that $X$ and $Y$ projectively centralize each other. 
In this case the category generated by $Y$ and $\chi$ centralizes $X\otimes
X^*$, so it is a proper subcategory of $\C$. If it is
not slightly degenerate, its M\"uger center contains more than
two simple objects, hence contains a nontrivial Tannakian subcategory. 
So we may assume that this subcategory is slightly degenerate, 
in which case we are done by the induction assumption. 

Thus, we may assume that $X$ and $Y$ do  not
projectively centralize each other. In this case 
Lemma \ref{key} tells us that $s_{X,Y}=0$.

Now, since $\C$ is slightly degenerate, by
Corollary \ref{ones}, we have 
$$
\sum_Y \frac{s_{X,Y}}{d_X}d_Y=0, 
$$
and all the nonzero terms in this sum correspond to either $\dim
Y=1$ (there are two such terms, both equal to $1$), 
or $\dim Y$ divisible by $p$, which gives terms divisible by
$p$. So we get that $2$ is divisible by $p$, a contradiction. 
\end{proof}

\section{Proof of Theorem \ref{Th2}}
\label{Sect 8}

We are going to prove the following theorem, which will easily imply
Theorem \ref{Th2}.

\begin{theorem}\label{integ} 
Let $\E$ be an integral non-degenerate braided category of dimension 
$p^aq^b$, where $p<q$ are primes, and $a,b$ nonnegative integers.
If $\E$ is not pointed, 
then it contains a Tannakian subcategory of the form $\Rep(G)$,
where $G$ is a cyclic group of prime order. 
\end{theorem}

Let us explain how Theorem \ref{integ} implies Theorem \ref{Th2}.
Let $\C$ be a fusion category of dimension $p^rq^s$, where $p<q$
are primes, and $r,s\ge 0$ are nonnegative integers. 
We prove that $\C$ is solvable by induction in $r+s$. 
We can assume that $\C$ is integral, because if 
not, then $\C$ is $\Bbb Z/2\Bbb
Z$-graded, so we are done by the induction assumption.
Also, we can clearly assume that $\C$ is not pointed. 
Clearly, the  center $\E:=\Z(\C)$ is not pointed. 
So the result follows by using Theorem \ref{integ}
and Proposition \ref{basicp2}. 

The rest of the section is devoted to the proof of Theorem \ref{integ}. 

\begin{proposition}\label{invob} Let $\E$ be an integral non-degenerate braided
category of dimension $p^aq^b$, $a+b>0$. Then $\E$
contains a nontrivial invertible object.
\end{proposition}
\begin{proof} By \cite{ENO} a fusion category of a prime power Frobenius-Perron dimension
is nilpotent and hence contains a nontrivial invertible object.
So we may assume that $a,\,b>0$. Assume the contrary, i.e. that $\E$ does not contain nontrivial 
invertible objects. By Theorem \ref{divis}, the squared dimensions of simple objects of $\E$
divide $p^aq^b$. Therefore,  
$\E$ must contain a simple object of dimension $p^r$,
$r>0$. Hence by Corollary \ref{primpow}, it contains a nontrivial
symmetric subcategory $\D$. By Theorem \ref{DeT}
$\D$ is super-Tannakian, and therefore by 
the usual Burnside theorem for finite groups (saying that a group
of order $p^aq^b$ is solvable), it must contain nontrivial 
invertible objects, which is a contradiction.  
\end{proof}

Consider now the subcategory $\B$ spanned by all invertible objects of $\E$. 
Proposition \ref{invob} implies that this subcategory is nontrivial. 
If $\B$ is non-degenerate, then by Theorem \ref{MuTh}, 
$\E=\B\boxtimes \B'$, where $\B'$ is another non-degenerate braided category, 
which is nontrivial (as $\E$ is not pointed), but 
has no nontrivial invertible objects. Thus, by Proposition \ref{invob},
this case is impossible. 

Therefore, $\B$ is degenerate. Consider the M\"uger center $\mZ$ of $\B$. 
It is a nontrivial pointed symmetric subcategory in $\E$. 
So if ${\rm FPdim}(\mZ)>2$,
we are done (as $\mZ$ necessarily contains a Tannakian
subcategory $\Rep(G)$, where $G$ is a cyclic group of prime order). 

It remains to consider the case ${\rm FPdim}(\mZ)=2$. In this
case, we must consider the additional possibility that 
$\mZ$ is the symmetric category ${\rm
SuperVec}$ of super vector spaces (in which case $p=2$). In this
situation, by Proposition \ref{almnon}(ii), $\mB=\mZ\boxtimes \D$, where $\D$ is
non-degenerate, so if $\D$ is nontrivial, by Theorem \ref{MuTh}
$\E=\D\boxtimes \D'$, where $\D'$ is another non-degenerate braided
category, whose subcategory of invertible objects is $\mZ$. 
Thus, it is sufficient to consider the case 
$\B=\mZ={\rm SuperVec}$. In this case,
let $\C\supset \mZ$ be the M\"uger centralizer
of $\mZ$. This category has dimension is $2^{a-1}q^b>2$, contains only two
invertible objects, and its M\"uger center is $\mZ={\rm
SuperVec}$, i.e. it is slightly degenerate. 
Therefore, Theorem \ref{integ} follows from 
the following proposition. 

\begin{proposition}\label{p2}
Let $\C$ be a slightly degenerate integral braided
category of dimension $2^rq^s>2$, where $q>2$ is a prime, and $r,s$
are nonnegative integers. Suppose that $\C$ contains only two
invertible objects. Then $\C$ contains a nontrivial Tannakian subcategory. 
\end{proposition}

\begin{proof}
It follows from Theorem \ref{divis} that there exists a
non-invertible simple object $Y$ of $\C$ whose dimension is a
power of $2$. Also, by Corollary \ref{odd}, $\C$ contains a
simple object $X$ of dimension $q^m$, $m>0$. Now the statement follows from 
Proposition \ref{almnontan}.
\end{proof}

\section{Applications}
\label{Sect 9}

\subsection{Fusion categories of dimension $pqr$}

\begin{proposition}\label{wgtt}
A weakly group-theoretical integral fusion category of square-free
dimension is group-theoretical. 
\end{proposition}

\begin{proof}
It follows from \cite[Corollary 5.3]{GNk}
that any nilpotent integral fusion category of
square-free dimension is automatically pointed. 
\end{proof} 

\begin{theorem}\label{pqr} Let $p<q<r$ be a triple of distinct primes.
Then any integral fusion category $\C$ of dimension $pqr$ is
group-theoretical.
\footnote{It is easy to see that any weakly integral but not
integral fusion category of dimension $pqr$ is solvable, because
in this case $p=2$, and the category has a $\Bbb Z/2\Bbb
Z$-grading with trivial component of dimension $qr$. Such
categories are not hard to classify, but we won't do it here.}
\end{theorem}

\begin{proof} 
By Proposition \ref{wgtt}, it suffices to
show that $\C$ is Morita equivalent to a nilpotent category,
i.e., is weakly group-theoretical. It suffices
to show that the category $\Z(\C)$ contains a
nontrivial Tannakian subcategory; then the result will follow from 
Proposition~\ref{basicp1} and Theorem~\ref{Th2}. 

\begin{lemma}\label{symsub} 
$\Z(\C)$ contains a nontrivial symmetric subcategory.
\end{lemma}

\begin{proof}
By Corollary \ref{primpow}, if $\Z(\C)$ contains a simple object of
prime power dimension, then it contains a nontrivial symmetric
subcategory and we are done. So it suffices to consider the case 
when $\Z(\C)$ does not contain simple objects of prime power
dimension. In this case, by Theorem \ref{divis}, the dimensions of simple objects 
of $\Z(\C)$ can be $1,pq,pr$, and $qr$.

Consider first the case when $\Z(\C)$ contains nontrivial
invertible objects. In this case, let $\B$ be the category
spanned by the invertible objects of $\Z(\C)$. If $\B$ is
degenerate, its M\"uger center is a nontrivial symmetric
category, and we are done. If $\B$ is non-degenerate, then by Theorem
\ref{MuTh}, $\Z(\C)=\B\boxtimes \B'$, where $\B'$ has no nontrivial 
invertible objects. It is clear that $\FPdim(\B')$ is not
divisible by one of the numbers $p^2,q^2,r^2$. Say it is $p^2$. 
Then by Theorem \ref{divis}, all nontrivial simple objects of
$\B'$ have dimension $qr$, which is a contradiction. 

Now consider 
the remaining case, i.e., when $\Z(\C)$ has no nontrivial
invertible objects. Then the dimensions of nontrivial simple objects
in $\Z(\C)$ are $pq, pr, qr$. Let $X$ be a simple object of $\Z(\C)$ of dimension
$qr$ (it is easy to see that it exists). We have the orthogonality relation
$$
\sum_{Y\in {\rm Irr}\Z(\C)}\frac{s_{X,Y}}{d_X}d_Y=0.
$$
Hence there exists $Y_0\in {\rm Irr}\Z(\C)$ of dimension $pq$
such that $s_{X,Y_0}\ne 0$
(otherwise the left hand side will be equal to $1$ modulo $r$). 
Since $\frac{s_{X,Y_0}}{d_X}$ and $\frac{s_{X,Y_0}}{d_{Y_0}}$ are algebraic integers, we have
that $\frac{s_{X,Y_0}}{pqr}$ is an algebraic integer; thus $\frac{s_{X,Y_0}}{d_X}$ is divisible
by $p$. Now we have
\begin{equation}\label{sumsq}
\sum_{Y\in {\rm Irr}\Z(\C)}\left| \frac{s_{X,Y}}{d_X}\right|^2=\frac{\FPdim(\C)^2}{d_X^2}=p^2.
\end{equation}
Notice that since $s_{X,Y}$ is a sum of roots of unity, 
every summand on the left hand side is a totally positive algebraic integer; the summand corresponding
to $Y=\be$ is 1 and the summand $s$ corresponding to $Y=Y_0$ is
an algebraic integer divisible by $p^2$. Thus there exists a Galois automorphism $g$
such that $g(s)\ge p^2$. Applying $g$ to both sides of
(\ref{sumsq}), we get a contradiction, as the left hand side is $\ge 1+p^2$.
\end{proof}

Now let us finish the proof of Theorem \ref{pqr}.
We are done in the case when $pqr$ is odd since a symmetric category of odd dimension
is automatically Tannakian. So let us assume that $p=2$.  
Let us prove that $\Z(\C)$ contains a nontrivial Tannakian subcategory. Assume not.
Then a maximal symmetric subcategory of $\Z(\C)$ is the category of super vector
spaces; let $\Z\subset \Z(\C)$ be its M\"uger
centralizer. Clearly, $\Z$ is slightly degenerate, and $\FPdim(\Z)=2q^2r^2$.

Assume first that $\Z$ has no invertible objects outside of $\Z'={\rm
SuperVec}$. In this case by Proposition~\ref{odd}
$\Z$ contains a non-invertible simple object $X$ of odd
dimension. We must have $d_X =q \mbox{ or } r$, since 
if $d_X = qr$ then $\chi\otimes X\ne X$ would also have
dimension $qr$. But  $\Z$ cannot contain two simple
$qr$-dimensional objects since then $\FPdim(\Z) \geq 1 +2q^2r^2$. 
Thus we are done by Proposition \ref{almnontan}. 

Now assume that $\Z$ does
contain invertible objects outside of
$\Z'={\rm SuperVec}$. In this case, consider the category $\B$
spanned by the invertible objects of $\Z$ of odd order. If $\B$
is degenerate, it contains a Tannakian subcategory and we get a
contradiction. If $\B$ is non-degenerate, then by Theorem \ref{MuTh}, 
$\Z=\B\boxtimes \B'$, and $\B'$ is a slightly degenerate category of dimension 
dividing $2q^2r^2$ with no invertible objects outside of ${\rm
SuperVec}$. By the above argument, either this category must contain a
simple object of dimension $q$ or $r$, in which case we are done
by Proposition \ref{almnontan}, or $\B'={\rm SuperVec}$. In the
latter case, $\Z(\C)=\B\boxtimes \hat \B$, where $\hat \B$ is a
4-dimensional integral non-degenerate braided category, hence $\Z(\C)$ is
pointed and there is nothing to prove. The theorem is proved.  
\end{proof}

\begin{corollary}\label{Hopfal}
Let $H$ be a semisimple Hopf algebra of dimension
$pqr$, where $p<q<r$ are primes. Then there exists 
a finite group $G$ of order $pqr$ and an exact factorization 
$G=KL$ of $G$ into a product of subgroups, such that 
$H$ is the split abelian extension $H(G,K,L,1,1,1)=
\Bbb C[K]\ltimes {\rm Fun}(L)$ associated
to this factorization. 
\end{corollary}

\begin{proof}
By Theorem \ref{pqr}, $H$ is group-theoretical.
Thus the result follows from the following lemma. 

\begin{lemma}\label{gt}
Let $H$ be a group-theoretical semisimple Hopf algebra 
of square-free dimension. Then $H$ a split abelian extension 
of the form $H(G,K,L,1,1,1)$. 
\end{lemma}

\begin{proof} Since $H$ is group theoretical, 
there exists a group $G$ and a 
cocycle $\omega\in Z^3(G,\Bbb C^*)$ such that 
$\Rep(H)$ is the group-theoretical category 
$\Vec_{G,\omega,K,\psi}$, of $(K,\psi)$-biequivariant
$(G,\omega)$-graded vector spaces (here $\psi$ is a
2-cochain on $K$ such that $d\psi=\omega|_K$). 
The fiber functor on $\Rep(H)$ corresponds to a module 
category $\M$ over $\Vec_{G,\omega,K,\psi}$ with 
one simple object. It is the category 
of $(G,\omega)$-graded vector spaces which are left-equivariant 
under $(K,\psi)$ and right equivariant under 
$(L,\phi)$, for another subgroup $L\subset G$ and 2-cochain
$\phi$ on $L$ such that $d\phi=\omega|_L$. 
The condition of having one simple object
implies that $KL=G$. Moreover, $\M$ is the category 
of projective representations of the group $K\cap L$ with
a certain 2-cocycle. But the group $K\cap L$ has square free
order, so its Sylow subgroups are cyclic, and thus 
this 2-cocycle must be trivial. So the one simple object condition
implies that $K\cap L=1$, so $G=KL$ is an exact factorization. 
Also, since $[\omega]|_K=[\omega]|_L=1$, we find that $\omega$
represents the trivial cohomology class. Finally, if we choose
$\omega=1$, then $\psi$ and $\phi$ are coboundaries, so we can
choose $\psi=1,\phi=1$. 

Thus, we have shown that both the category $\Rep(H)$ and the fiber functor on it
attached to $H$ are the same as those for $H(G,K,L,1,1,1)$.
This implies that $H=H(G,K,L,1,1,1)$, 
as desired.       
\end{proof}

\end{proof}

\subsection{Classification of semisimple Hopf algebras of dimension $pq^2$}

In this section we classify semisimple Hopf algebras of dimension
$pq^2$, generalizing the results of \cite{G,N4,N5}. 

Let $p,\, q$  be distinct primes. 

\begin{proposition} 
\label{pq2}(\cite{JL})
Every semisimple Hopf algebra $H$ of dimension $pq^2$ is group-theoretical.
\end{proposition}
\begin{proof}
By Theorem~\ref{Th2} $\Rep(H)$ is either an extension or an equivariantization
of a fusion category of smaller dimension. 

Suppose  $\Rep(H)$ is an extension.  Then  $H$ contains a
central Hopf subalgebra $K$ of prime dimension,
and therefore it is an extension of the form
\[
\Bbb C \to K \to H \to L \to \Bbb C,
\]
where $L$ is a Hopf algebra with $\dim L$ being a product 
of two primes.
If $L$ is cocommutative then $H$ is a Kac algebra, hence 
group-theoretical (\cite{N1}).
Otherwise, $\dim L =pq$, and $L$ must be commutative by  \cite{EG2},
so the trivial component of $\Rep(H)$ is pointed of dimension $pq$,
hence $\Rep(H)$ must be pointed. 
 
Suppose now that $\Rep(H)$ is an equivariantization, i.e.,  $\Rep(H) = \C^G$
for a cyclic group $G$ of prime order. By \cite[Corollary 3.6]{Nk2} $\C^G$
is group-theoretical if and only if there is a $G$-invariant indecomposable 
$\C$-module category $\M$ such that the dual category $\C^*_\M$ is pointed.
Clearly, such a category always exists if $\C$ itself 
is pointed (take $\M =\C$).
By \cite{EGO}, the only non-pointed possibility 
for $\C$ is the representation category of
a non-commutative group algebra of dimension $pq$. But this category has
a unique (and hence $G$-invariant) fiber functor. The dual with respect
to this functor is pointed, and so $\Rep(H)$ is group-theoretical. 
\end{proof}

\begin{corollary}
A semisimple Hopf algebra of dimension $pq^2$ is either a Kac algebra, 
or a twisted group algebra (by a twist corresponding to 
the subgroup $(\Bbb Z/q\Bbb Z)^2$), or the dual of a twisted
group algebra. 
\end{corollary}

\begin{proof}
The situation is the same as in the proof of 
Lemma \ref{gt}, except that now the group $K\cap L$ 
does not have to be trivial. The condition on this group is that 
it must have a non-degenerate 2-cocycle. The only case when this
group can be nontrivial is when $K=G$, $L=(\Bbb Z/q\Bbb Z)^2$, 
or $L=G$, $K=(\Bbb Z/q\Bbb Z)^2$, which implies the statement.  
\end{proof}

\begin{remark}
1. There exist integral fusion categories of dimension $pq^2$
which are {\em not} group-theoretical, e.g., certain
Tambara-Yamagami categories \cite{TY, ENO}. 
By Proposition~\ref{pq2} they are not equivalent to
representation categories of
semisimple Hopf algebras. A classification of 
such categories and another proof of 
Proposition \ref{pq2} based on this classification will appear in
\cite{JL}.   
\end{remark}

\subsection{Semisimple quasi-Hopf algebras of dimension $p^rq^s$}

\begin{proposition}\label{gplike} Any semisimple quasi-Hopf
(in particular, Hopf) algebra of dimension $p^rq^s>1$ 
(where $p,q$ are primes) has a nontrivial 1-dimensional
representation. Therefore, any semisimple Hopf
algebra of dimension $p^rq^s>1$ has a nontrivial group-like
element. 
\end{proposition}

\begin{proof}
The first statement follows from Theorem \ref{Th2} and
Proposition \ref{basicp2-bis}. The second statement follows from
the first one by taking the dual.  
\end{proof}

\subsection{Simple fusion categories} 

\begin{definition}
A fusion category is called {\em simple} if it has no non-trivial
proper fusion subcategories.
\end{definition}

Clearly, a pointed fusion category is simple iff it is equivalent
to $\Vec_{G,\omega}$ for a cyclic group $G$ of prime order. 

\begin{proposition} 
If $\C$ is a weakly group-theoretical simple fusion category
which is not pointed, 
then $\C=\Rep(G)$, where $G$ is a non-abelian finite simple group.
\end{proposition}

\begin{proof} Consider the  center $\Z(\C)$. It contains a nontrivial
Tannakian subcategory $\Rep(G)$ with $|G|\le \FPdim(\C)$ 
that maps to $\C$. If it maps to $\Vec$,
we get a $G$-grading on $\C$, and we are done. Otherwise, the image of
$\Rep(G)$ in $\C$ is a nontrivial fusion 
subcategory. So it must be the whole $\C$, and by
dimension argument the functor $\Rep(G)\to \C$ is an equivalence, so
we are done.
\end{proof} 

\subsection{Simple categories of dimension 60}

By Theorem~\ref{Th2} and Theorem \ref{pqr}, 
all weakly integral fusion categories of dimension $ <60$
are solvable. Thus, the only simple weakly 
integral fusion categories of dimension $ <60$ 
are the categories $\Vec_{G,\omega}$ for cyclic groups $G$ 
of prime order.

The goal of this subsection
is to prove the following result:

\begin{theorem}\label{60}
Let $\C$ be a simple fusion category of dimension $60$.  Then
$\C \cong \Rep(A_5)$, where $A_5$ is the alternating group of order 60.
\end{theorem}

\begin{proof} 
It is clear that the category $\C$ is integral, since otherwise it would
contain a nontrivial subcategory $\C_{\rm ad}$, see \cite[Proposition
8.27]{ENO}.

Consider first the case when $\Z(\C)$ is not simple. In this case
$\Z(\C)$ contains  a non-trivial subcategory $\D$ of dimension
$\leq 60$ (if the dimension of $\D$ is $>60$, we will replace $\D$ with its
M\"uger centralizer $\D'$). 

Let $F:\Z(\C)\to \C$ be the forgetful
functor. The fusion subcategory $F(\D) \subset
\C$ is nontrivial, since $\C$ has
no nontrivial gradings.  
Since $\C$ is simple, this means that 
$\D$ has dimension exactly $60$ and 
$F: \C \cong  \D$ is an equivalence, so $\C$ is a braided category. 
Clearly, $\C$ contains objects of prime power dimension, 
since for any representation $60=1+\sum n_i^2$, $n_i>1$, some
$n_i$ has to be a prime power. Thus, 
$\C$ cannot be non-degenerate by Corollary~\ref{primpow}.
Therefore, $\C$ must be symmetric, i.e. $\C \cong \Rep(G)$ for 
a simple group $G$. Since $A_5$ is the unique simple group of order $60$,
we obtain $G \cong A_5$. 

Now let us assume that $\Z(\C)$ is simple. 

\begin{lemma} 
The dimensions of nontrivial simple objects in $\Z(\C)$ are 
among the numbers $6,10,15,30$.
\end{lemma}

\begin{proof} 
Theorem \ref{divis} (i) 
and Corollary \ref{primpow} show that possible dimensions 
of nontrivial simple objects of $\Z(\C)$ are $6,10,12,15,20,30$.
Thus we just need to exclude the dimensions $60/p$, where $p=3$ or
$5$. In both cases the argument
is parallel to the proof of Lemma \ref{symsub}.
Namely, let $X$ be a simple object of $\Z(\C)$ of dimension
$60/p$. We have the orthogonality relation
$$
\sum_{Y\in {\rm Irr}\Z(\C)}\frac{s_{X,Y}}{d_X}d_Y=0.
$$
Hence there exists $Y_0\in {\rm Irr}\Z(\C)$ of dimension
divisible by $p$ such that $s_{X,Y_0}\ne 0$
(otherwise the left hand side will be equal to $1$ modulo $15/p$). 
Since $\frac{s_{X,Y_0}}{d_X}$ and $\frac{s_{X,Y_0}}{d_{Y_0}}$ 
are algebraic integers, we have
that $\frac{s_{X,Y_0}}{60}$ is an algebraic integer; 
thus $\frac{s_{X,Y_0}}{d_X}$ is divisible
by $p$. The rest of the argument is word for word as in
the proof of Lemma \ref{symsub}.  
\end{proof}

There is only one decomposition of 
$60-1=59$ into the sum of numbers $6,10,15,30$,
namely $59=15+10+10+6+6+6+6$. 
It follows that the object $I(\be)\in \Z(\C)$ 
(where $I:\C\to \Z(\C)$ is the induction functor) 
has precisely
8 simple direct summands, 
hence $\dim \Hom(I(\be),I(\be))\ge 8$. Then \cite[Proposition 5.6]{ENO}
implies that the category $\C$ contains at least 8 simple
objects. Hence $\C$ contains a 
nontrivial simple object with the square of dimension 
less or equal to $\frac{59}{7}<9$; thus this object is
of dimension 2. But it is proved in 
\cite{NR} that an integral simple fusion category cannot  
contain a simple object of dimension $2$. 
\footnote{To be more
precise, the argument in \cite{NR} is for comodule categories 
over finite dimensional cosemisimple Hopf algebras, but it uses only the
Grothendieck ring arguments, and therefore applies verbatim to the
case of fusion categories.} The theorem is proved.
\end{proof}

\begin{corollary}
Up to isomorphism, the only semisimple Hopf algebras of dimension $60$ without 
non-trivial Hopf algebra quotients 
are the group algebra $\Bbb C[A_5]$  
and its twisting deformation  $\Bbb C[A_5]_J$ constructed in \cite{Nk1}.
\end{corollary}
\begin{proof}
Let $H$ be such a Hopf algebra. Then $\Rep(H)$ is simple and hence $\Rep(H) \cong \Rep(A_5)$
by Theorem~\ref{60}. Therefore, $H$ is a twisting deformation  of
$\Bbb C[A_5]$. By \cite{EG3} twisting
deformations of a group algebra $\Bbb C[G]$ 
correspond to non-degenerate $2$-cocycles on subgroups
of $G$. Each Sylow $2$-subgroup of $A_5$ admits a unique 
(up to cohomological equivalence)
non-degenerate $2$-cocycle. All such subgroups are conjugate  
and so the corresponding
twisting deformations are gauge equivalent and yield the example of \cite{Nk1}.
\end{proof}

\begin{remark}
The property of a Hopf algebra having no non-trivial 
quotients is {\em stronger} than that
of having no nontrivial normal Hopf subalgebras. In particular,
there exist other semisimple Hopf algebras of dimension $60$ 
without nontrivial normal Hopf subalgebras.
Such are, e.g., $\Bbb C[A_5]_J^*$ 
and  the example constructed in \cite[4.2]{GN}.
\end{remark}

\subsection{Non-simple fusion categories of dimension 60}

\begin{theorem}\label{di60}
Any fusion category $\C$ of dimension $60$ is weakly group-theoretical.
\end{theorem}

\begin{proof}
If $\C$ is simple, then the result follows from the previous
subsection. So let us consider the case when $\C$ is not simple. 
In this case, we may assume that $\C$ is integral (otherwise 
$\C$ is $\Bbb Z/2\Bbb Z$ graded, hence solvable). 

Let $\Z(\C)$ be the  center of $\C$. 
It suffices to show that $\Z(\C)$ contains a nontrivial Tannakian
subcategory. 

\begin{lemma}\label{symc}
$\Z(\C)$ contains a nontrivial symmetric subcategory. 
\end{lemma}

\begin{proof}
Recall that in $\Z(\C)$ we have a standard commutative algebra $A=I(\bold
1)$, whose category of modules is $\C$. 
We may assume that $A$ does not contain nontrivial invertible
objects; otherwise $\C$ is faithfully graded by a nontrivial
group, and we are done. 

Let $\D\subset\C$ be a nontrivial proper fusion subcategory, of 
codimension $1<d<60$ (i.e., dimension $60/d$). 
We claim that there exists an algebra $B\subset A$ in $\Z(\C)$ of dimension
$d$. Indeed, consider the category $\E$ of pairs 
$(X,\eta)$, where $X$ is an object of $\C$, and $\eta: \otimes_{\C,\D}\to
\otimes_{\D,\C}$ is a functorial isomorphism satisfying the hexagon
relation (where $\otimes _{\C,\D}$, $\otimes _{\D,\C}$ 
are the tensor product functors
$\C\times \D\to \C$, $\D\times \C\to \C$). In other words, 
$\E$ it is the dual category to 
$\C\boxtimes \D^{op}$ with
respect to the module category $\C$. Thus, we have
a diagram of tensor functors $\Z(\C)\to \E\to \C$,
whose composition is the standard forgetful functor 
$\Z(\C)\to \C$. Denote the functor $\Z(\C)\to \E$ by $F$. 
This functor is surjective, as it is dual to the inclusion
$\C\to \C\boxtimes \D^{op}$. Let $F^\vee$ be the adjoint functor
to $F$. Then $B=F^\vee(\bold
1)\subset A$ is the desired subalgebra. 
 
The existence of the algebra $B$ implies that $I(\bold 1)\in
\Z(\C)$ contains a simple object of prime power dimension. 
Indeed, if not, then the
dimensions of simple objects can be $1,6,10,12,15,20,30$. 
On the other hand, the dimension of $B$ is some divisor $d$ of $60$,
$1<d<60$. Thus, we have $d-1=\sum n_i$, where 
$n_i=6,10,12,15,20,30$. It is checked by inspection that this is
impossible.   
\end{proof}

Consider the nontrivial symmetric category contained in $\Z(\C)$.
If its dimension is bigger than 2, it contains a nontrivial Tannakian
subcategory, and we are done. So it remains to consider the case 
when the only nontrivial symmetric subcategory of $\Z(\C)$ is
SuperVec. We make this assumption in the remainder of the
section. 

Let $\Z$ be the M\"uger centralizer of
the subcategory SuperVec. Proposition \ref{almnon}(i) 
implies that if $\Z$ contains a simple object of odd prime power
dimension, then it contains a nontrivial Tannakian subcategory. 
Let us now consider the case of even prime power dimension. 

\begin{lemma}\label{2and4}
(i) If $\Z$ contains a 2-dimensional simple object $X$, then it
contains a nontrivial Tannakian subcategory. 

(ii) If $\Z$ contains a 4-dimensional simple object $X$, then it
contains a nontrivial Tannakian subcategory. 
\end{lemma}

\begin{proof}
(i) Let $\B$ be the category spanned by the invertible objects 
of $\Z$. By Proposition \ref{almnon}(ii), $\B={\rm SuperVec}\boxtimes
\B_0$, where $\B_0$ is a non-degenerate pointed category.
Then $\Z=\B_0\boxtimes \B_0'$, and $\B_0'$ contains only two
invertible objects, $\bold 1$ and $\chi$ (the generator of
SuperVec). Clearly, $X\otimes X^*\in \B_0'$, and $\chi\otimes
X\ne X$ by Proposition \ref{almnon}(i), so $X\otimes X^*=\bold
1\oplus Y$, where $Y$ is 3-dimensional. Thus we 
are done by Proposition \ref{almnontan}.

(ii) Arguing as in (i), we see that $X\otimes X^*=\bold 1+...$, 
where $...$ is a direct sum of simple objects of $\Z$ of dimension $>1$.
Moreover, at least one of these dimensions must be odd, since the total is 15. 
If there is an object of dimension $3$ or $5$, then we are
done by Proposition \ref{almnontan}. Otherwise, 15 is the
smallest odd dimension that can occur (by Theorem \ref{divis}), 
so we must have $X\otimes X^*=\bold 1\oplus Y$, where $Y$ has
dimension 15. Then we would have $s_{XX^*}=\lambda+15\mu$, where $\lambda,\mu$
are roots of unity. On the other hand, $s_{XX^*}$ is divisible by
$d_X=4$. Thus $\lambda-\mu$ is divisible by $4$. So
$\lambda-\mu=0$, and 
$X$ projectively centralizes its dual, hence itself. Thus $Y$ centralizes
itself, so it generates a symmetric category, which must contain 
a nontrivial Tannakian subcategory. 
\end{proof} 

Lemma \ref{2and4} shows that 
it suffices to prove that $\Z$ contains a simple
object of prime power dimension; then we will be done by 
Proposition \ref{almnontan}.

To show this, let $A_+\subset A$ be $\chi\otimes K$, where $K$ is
the kernel of $c^2-1$ (squared braiding minus one) on
$\chi\otimes A$. Then $A_+$ is a subalgebra of $A$ contained in
$\Z$, and by the argument at the end of the proof of Lemma \ref{symc}, $A_+$
contains a simple object of prime power dimension.
The theorem is proved. 
\end{proof} 

\section{Relation  with previous results on classification of semisimple Hopf 
algebras and fusion categories.}
\label{Sect 10}

1. Theorem \ref{Th2} for $s=0$, i.e. for fusion categories of prime power
dimension, follows from \cite{ENO}, where it is shown that any
such category is cyclically nilpotent. For $r=s=1$ (i.e. for
fusion categories of dimension $pq$), Theorem \ref{Th2} follows
from the paper \cite{EGO}, where such categories are classified
(we note that the main results of \cite{EG2} and \cite{EGO} are
trivial consequences of Theorem \ref{Th2}).

2. Semisimple Hopf algebras of small dimension were studied
extensively in the literature, see \cite[Table 1]{N2} for the
list of references. In particular, in the monograph \cite{N2} it
is shown that semisimple Hopf algebras of dimension $< 60$ are
either upper or lower semisolvable up to a cocycle twist, and in 
\cite{N6} it is shown that any semisimple Hopf algebra of
dimension $<36$ is group-theoretical (this is not true for
dimension 36, see \cite{Nk2}). Our Theorems~\ref{Th2} 
and \ref{pqr} along with \cite{DGNO1} further
describe the structure (of representation categories) of such
Hopf algebras in group-theoretical terms.

Numerous classification results and non-trivial examples of Hopf
algebras of dimension $pq^n$ are obtained in \cite{Ma, EG2, F, G,
N2, N3, N4, N5, IK}.  Some of these results use an assumption of
Hopf algebras involved being of Frobenius type. Our
Theorem~\ref{Th1} shows that this assumption is always satisfied.

Semisimple Hopf algebras of dimension $pqr$ we studied in
\cite{AN, N3}. In particular, Hopf algebras of dimension $30$ and
$42$ were classified as Abelian extensions (Kac algebras).  Our
Corollary~\ref{Hopfal} is a generalization of these results. It
implies that \cite[Theorem 4.6]{N3} can be used to obtain a complete 
classification of such Hopf algebras.

\section{Questions}
\label{Sect 11}

We would like to conclude the paper with two questions. 

{\bf Question 1.} Does there exist a fusion category 
that does not have the strong Frobenius property?

{\bf Question 2.} Does there exist a weakly integral fusion category which is
not weakly group-theoretical?

\bibliographystyle{ams-alpha}

\begin{thebibliography}{A} 

\bibitem[AN]{AN}
N.~Andruskiewitsch, S.~Natale,
\textit{Examples of self-dual Hopf algebras},
J\. Math.\ Sci.\ Univ.\ Tokyo \textbf{6} (1999), no. 1, 181--215.

\bibitem[BK]{BK} B.~Bakalov, A.~Kirillov Jr.,
\textit{Lectures on Tensor categories and modular functors}, 
AMS, (2001).

\bibitem[B]{B} A.~Brugui\`{e}res,
\textit{Cat\'egories pr\'emodulaires, modularization et invariants
des vari\'et\'es de dimension $3$}, Mathematische Annalen,
\textbf{316} (2000), no. 2, 215--236.

\bibitem[DMNO]{DMNO} A.~Davydov, M.~M\"uger, D.~Nikshych,
and V.~Ostrik, 
\textit{\'Etale algebras in braided fusion categories}, in preparation.

\bibitem [D1]{D1} P.~Deligne,
\textit{Cat\'egories tannakiennes}, in : \textit{The Grothendieck
Festschrift}, 
Vol. II, Progr. Math., \textbf{87}, Birkh\"auser, Boston, MA, 1990, 111--195.

\bibitem [D2]{D2} P.~Deligne,
\textit{Cat\'egories tensorielles}, Moscow Math. Journal \textbf{2}
(2002) no.2, 227 -- 248.

\bibitem [DM]{DM} J.~Donin, A.~Mudrov, 
\textit{Dynamical Yang-Baxter equation and quantum vector bundles}, 
Commun.\ Math.\ Phys.\ \textbf{254}, 719--760 (2005).

\bibitem[DGNO1]{DGNO1} V.~Drinfeld, S.~Gelaki, D.~Nikshych,
and V.~Ostrik,
\textit{Group-theoretical properties of nilpotent
modular categories}, eprint arXiv:0704.0195v2 [math.QA].

\bibitem[DGNO2]{DGNO2} V.~Drinfeld, S.~Gelaki, D.~Nikshych,
and V.~Ostrik, 
\textit{On braided fusion categories I}, eprint arXiv:0906.0620  [math.QA].

\bibitem[EG1]{EG1} P.~Etingof, S.Gelaki, 
\textit{Some properties of finite-dimensional semisimple  Hopf algebras}
Math.\ Research Letters \textbf{5}, 191--197 (1998).

\bibitem[EG2]{EG2} P.~Etingof, S.Gelaki, 
\textit{Semisimple Hopf algebras of dimension $pq$ are trivial},
J. Algebra \textbf{210} (1998), no. 2, 664--669. 

\bibitem[EG3]{EG3} P.~Etingof, S.Gelaki, 
\textit{The classification of triangular semisimple 
and cosemisimple Hopf algebras over an algebraically closed
field}, Internat. Math. Res. Notices (2000), no. 5, 223--234.

\bibitem[EGO]{EGO}
P. ~Etingof, S.~Gelaki, and V.~Ostrik, 
\textit{Classification of fusion categories of dimension pq.}  
Int. Math. Res. Notices  2004,  no. 57, 3041--3056.

\bibitem[ENO]{ENO} P.~Etingof, D.~Nikshych, and V.~Ostrik,
\textit{On fusion categories}, Annals of Mathematics 
\textbf{162} (2005), 581--642.

\bibitem[EO]{EO} P.~Etingof, V.~Ostrik,
\textit{Finite tensor categories},
Moscow Math.\ Journal \textbf{4} (2004), 627--654.

\bibitem[F]{F} N.~Fukuda, 
\textit{Semisimple Hopf algebras of dimension $12$},
Tsukuba J.\ Math.\ \textbf{21} (1997), no. 1, 43--54. 



\bibitem[GN]{GN} C.~Galindo, S.~Natale, 
\textit{Simple Hopf algebras and deformations of finite groups},  
Math.\ Res.\ Lett.\  \textbf{14}  (2007),  no. 6, 943--954.

\bibitem[G]{G} S.~Gelaki, 
\textit{Quantum groups of dimension $pq\sp 2$},  
Israel J.\ Math.\ \textbf{102} (1997), 227--267. 

\bibitem[GNN]{GNN} S.~Gelaki, D.~Naidu, and D.~Nikshych,
\textit{Centers of graded fusion categories}, 
eprint  arXiv:0905.3117 [math.QA].

\bibitem[GNk]{GNk} S.~Gelaki, D.~Nikshych,
\textit{Nilpotent fusion categories}, 
Adv.\ Math.\ \textbf{217} (2008), 1053--1071.

\bibitem[IK]{IK} M.~Izumi, H.~Kosaki, 
\textit{Kac algebras arising from composition of subfactors: general theory and classification},  
Mem.\ Amer.\ Math.\ Soc.\  \textbf{158}  (2002),  no. 750.

\bibitem[JL]{JL} D.~Jordan and E.~Larson, 
\textit{On the classification of certain fusion categories}
eprint arXiv:0812.1603v2 [math.QA].

\bibitem[K]{K} A.~Kirillov, Jr.,
\textit{Modular categories and orbifold models II},
arXiv:math/0110221.

\bibitem[KR]{KR} L.~Kong, I.~Runkel,
\textit{Cardy algebras and sewing constraints, I}, arXiv:0807.3356.

\bibitem[Ma]{Ma} A.~Masuoka,
\textit{Some further classification results on semisimple Hopf algebras},
 Comm. Algebra  \textbf{24}  (1996),  no. 1, 307--329.

\bibitem [M1]{M1} M.~M\"uger, 
\textit{Galois theory for braided tensor categories
and the modular closure}, Adv. Math. {\bf 150} (2000), no. 2, 151--201.

\bibitem [M2]{M2} M.~M\"uger, 
\textit{From subfactors to categories and topology I. Frobenius algebras in 
and Morita equivalence of tensor categories}, 
J. Pure Appl. Algebra \textbf{180} (2003), 81--157.

\bibitem [M3]{M3} M.~M\"uger, 
Galois extensions of braided tensor categories and braided
crossed G-categories. J. Alg. 277, 256--281 (2004).

\bibitem [M4]{M4} 
M. M\"uger, On the structure of modular categories, Proc. London
Math. Soc. (3) 87 (2003), no. 2, 291--308.

\bibitem [MW]{MW} S.~Montgomery, S.~Witherspoon,
\textit{Irreducible representations of crossed products},
J.\ Pure Appl.\ Algebra,  \textbf{129}  (1998),  no. 3, 315--326.

\bibitem[NN]{NN} D.~Naidu and D.~Nikshych,
\textit{Lagrangian subcategories and braided tensor equivalences
of twisted quantum doubles of finite groups}, 
Comm.\ Math.\ Phys.\ \textbf{279} (2008), 845--872.

\bibitem [NR]{NR} W.~Nichols, M.~Richmond,
\textit{The Grothendieck group of a Hopf algebra},
J. Pure Appl. Algebra  \textbf{106}  (1996),  no. 3, 297--306. 

\bibitem[N1]{N1} S.~Natale,
\textit{On group theoretical Hopf algebras and exact factorizations of finite groups},
J.\ Algebra  \textbf{270}  (2003),  no. 1, 199--211. 

\bibitem[N2]{N2} S.~Natale, 
\textit{Semisolvability of semisimple Hopf algebras of low dimension},
Mem.\ Amer.\ Math.\ Soc.\  \textbf{186}  (2007),  no. 874.

\bibitem[N3]{N3} S.~Natale,
\textit{On semisimple Hopf algebras of dimension $pq\sp 2$},  
J.\ Algebra  \textbf{221}  (1999),  no. 1, 242--278

\bibitem[N4]{N4} S.~Natale,
\textit{On semisimple Hopf algebras of dimension $pq\sp 2$. II},  
Algebr.\ Represent.\ Theory  \textbf{4}  (2001),  no. 3,

\bibitem[N5]{N5} S.~Natale,
\textit{On semisimple Hopf algebras of dimension $pq\sp r$},  
Algebr. Represent. Theory  \textbf{7}  (2004),  no. 2, 173--188.

\bibitem[N6]{N6} S.~Natale,
\textit{Hopf algebra extensions of group algebras and Tambara-Yamagami categories},
arXiv:0805.3172.

\bibitem[Nk1]{Nk1} D.~Nikshych,
\textit{$K_0$-rings and twisting of finite dimensional semisimple Hopf 
algebras}, Commun. Alg. \textbf{26} (1998), 321--342.

\bibitem[Nk2]{Nk2} D.~Nikshych,
\textit{Non group-theoretical semisimple Hopf algebras from
group actions on  fusion categories}, Selecta Math.\ \textbf{14} (2008), 145-161.

\bibitem [O1]{O1} V.~Ostrik,
\textit{Module categories, weak Hopf algebras and modular invariants},
Transform. Groups, \textbf{8} (2003), 177-206.

\bibitem [O2]{O2} V.~Ostrik,
\textit{Module categories over the Drinfeld double of a finite group},
Int.\ Math.\ Res.\ Not. (2003) no.\ 27, 1507--1520.

\bibitem[T]{T} D.~Tambara, 
\textit{Invariants and semi-direct products for finite group
actions on tensor categories}, 
J.~Math.~Soc.~Japan \textbf{53} (2001), no. 2, 429--456.

\bibitem[TY]{TY} D.~Tambara, S.~Yamagami, 
\textit{Tensor categories with fusion rules of self-duality for finite abelian groups}, 
J.~Algebra \textbf{209} (1998), no. 2, 692--707.

\end{thebibliography}

\end{document}